\documentclass{ifacconf}

\usepackage{graphicx}      %
\usepackage{natbib}        %
\usepackage{tikz}
\usetikzlibrary{shapes,arrows}
\usepackage{amsmath}
\usepackage{thmtools}
\usepackage{amssymb}
\usepackage{amsfonts}
\usepackage{centernot}
\usepackage{arydshln}
\usepackage{float}
\usepackage{optidef}
\usepackage{caption}
\usepackage{subcaption}
\usepackage{algorithm}
\usepackage{algpseudocode}
\usepackage{enumerate}
\usepackage{comment}

\usepackage{pgfplots}
\pgfplotsset{compat=newest}
\newlength\figurewidth \newlength\figureheight
\newcommand{\transp}{^{\mathsf{T}}}

\makeatletter
\let\cl@part\@empty  %
\makeatother
\definecolor{ao(english)}{rgb}{0.0, 0.0, 1.0}
\usepackage[nameinlink]{cleveref}
\crefformat{equation}{\textup{#2(#1)#3}}
\crefrangeformat{equation}{\textup{#3(#1)#4--#5(#2)#6}}
\crefmultiformat{equation}{\textup{#2(#1)#3}}{ and \textup{#2(#1)#3}}
{, \textup{#2(#1)#3}}{, and \textup{#2(#1)#3}}
\crefrangemultiformat{equation}{\textup{#3(#1)#4--#5(#2)#6}}%
{ and \textup{#3(#1)#4--#5(#2)#6}}{, \textup{#3(#1)#4--#5(#2)#6}}{, and \textup{#3(#1)#4--#5(#2)#6}}

\Crefformat{equation}{#2Equation~\textup{(#1)}#3}
\Crefrangeformat{equation}{Equations~\textup{#3(#1)#4--#5(#2)#6}}
\Crefmultiformat{equation}{Equations~\textup{#2(#1)#3}}{ and \textup{#2(#1)#3}}
{, \textup{#2(#1)#3}}{, and \textup{#2(#1)#3}}
\Crefrangemultiformat{equation}{Equations~\textup{#3(#1)#4--#5(#2)#6}}%
{ and \textup{#3(#1)#4--#5(#2)#6}}{, \textup{#3(#1)#4--#5(#2)#6}}{, and \textup{#3(#1)#4--#5(#2)#6}}

\crefdefaultlabelformat{#2\textup{#1}#3}
\newtheorem{theorem}{Theorem}
\newtheorem{proposition}[theorem]{Proposition}

\newtheorem{lemma}{Lemma}

\makeatletter
\def\@ssect#1#2#3#4#5#6{%
  \@tempskipa #4\relax
  \ifdim \@tempskipa>\z@
    \begingroup
      #5{%
        \@hangfrom{\hskip #2}%
          \interlinepenalty \@M \head@format{#1}{#6}\@@par}%
    \endgroup
  \else
    \def\@svsechd{#5{\hskip #2\relax \head@format{#1}{#6}}}%
  \fi
  \@xsect{#4}}
\makeatother

\begin{document}

\tikzstyle{block} = [draw, fill=white, rectangle, 
    minimum height=3em, minimum width=6em]
\tikzstyle{sum} = [draw, fill=white, circle, node distance=1cm]
\tikzstyle{input} = [coordinate]
\tikzstyle{output} = [coordinate]
\tikzstyle{junction} = [coordinate]
\tikzstyle{pinstyle} = [pin edge={to-,thin,black}]

\begin{frontmatter}

\title{Model-Free Disturbance Observer with Online Modification: Listening to MFDOOM} 

\author[First]{Nadav Barak} 
\author[Second]{Christian Grussler} 

\address[First]{Technion – Israel Institute of Technology, Faculty of Mechanical Engineering, Haifa 3200003 Israel (e-mail: nadav.barak@campus.technion.ac.il).}
\address[Second]{Technion – Israel Institute of Technology, The Stephen B. Klein Faculty of Aerospace Engineering, Haifa 3200003 Israel (e-mail: cgrussler@technion.ac.il)}

\begin{abstract}                %
Data-Enabled Predictive Control (DeePC) has recently emerged as a framework for controlling unknown systems from data. However, its performance relies on the relevance of the collected data, and as such, disturbances lead to inevitable errors. This paper addresses this problem by proposing an augmentation of DeePC using Model-Free Disturbance Observer with Online Modification (MFDOOM). The method corrects output predictions based on previous prediction errors using a dedicated continuously updated Hankel matrix. We compare our method, both theoretically and through simulation, to other recent algorithms designed for time-varying systems in the DeePC framework. It is shown that for disturbances that can be modeled as the output of an autonomous linear time-invariant system, this approach can reduce tracking error and online-update burden compared with existing online DeePC variants.
\end{abstract}

\begin{keyword}
Data-Enabled Predictive Control, Autonomous Systems, Disturbance Observer %
\end{keyword}

\end{frontmatter}

\section{INTRODUCTION}
Data-Enabled Predictive Control (DeePC) has gained significant attention as a model-free alternative to Model Predictive Control (MPC) \citep{ODeePC,MDeePC,ChenDeePC2025,elokda2021data,huang2021decentralized,huang2023robust,zieglmeier2025data}. 
Rather than using an explicit model, DeePC predicts future behavior using Hankel matrices constructed from input-output data \citep{DeePC}. 
Because unmodeled or insufficiently represented dynamics can degrade performance, regularization has been proposed to improve robustness, particularly for nonlinear or linear time-varying (LTV) systems \citep{DeePC}. 
Related online variants include Online DeePC (ODeePC), which refreshes the Hankel matrices with new measurements \citep{ODeePC}, and Online Reduced-Order DeePC, here referred to as MDeePC, which updates a mosaic Hankel matrix only when new trajectories enrich the implicit model \citep{MDeePC}.

When disturbances are generated by an autonomous linear system, e.g., sinusoids, they effectively alter the measured plant behavior. 
Since such disturbances may appear or disappear, the resulting prediction error is time-varying. As a result, standard DeePC may under perform, while conservative regularization can degrade nominal tracking \citep{Regularization}. Modifying the nominal behavior matrices (i.e. ODeePC) or expanding them (MDeePC) partially address this issue, but coupling disturbance adaptation to the baseline predictor results in several drawbacks, such as needless noise injection or increased online computation.

\begin{figure}[ht]\centering
\begin{tikzpicture}
	
	\begin{axis}[%
		width=0.85\columnwidth,   %
		height=0.3\columnwidth,  %
		scale only axis,
		xmin=0,
		xmax=15,
		ymin=-3,
		ymax=3.3,
		xlabel={Time [s]},
		ylabel={Velocity $[rad/s]$},
		axis background/.style={fill=white},
		axis x line*=bottom,
		axis y line*=left,
		xmajorgrids,
		ymajorgrids,
		legend style={
			at={(0.45,1.05)},
			anchor=south,
			legend columns=-1,
			legend cell align=left,
			align=left,
			draw=white!15!black
		}
		]
		\addplot [color=black]
		table[row sep=crcr]{%
			0	0\\
			0.300000000000001	7.93328069903509e-10\\
			0.4	0.309016850580837\\
			0.449999999999999	0.453990288550745\\
			0.5	0.587784978604651\\
			0.550000000000001	0.707106452255889\\
			0.6	0.809016617406737\\
			0.65	0.891006109952157\\
			0.699999999999999	0.951056072898103\\
			0.75	0.987687879211414\\
			0.800000000000001	0.999999534603861\\
			0.85	0.987687881358799\\
			0.9	0.951056073502235\\
			0.949999999999999	0.891006109478651\\
			1	0.809016617241598\\
			1.05	0.707106451874125\\
			1.1	0.587784978447669\\
			1.15	0.453990287710598\\
			1.25	0.156434392028871\\
			1.4	-0.309016850333803\\
			1.45	-0.453990288607482\\
			1.5	-0.58778497892539\\
			1.55	-0.707106451915733\\
			1.6	-0.809016617808386\\
			1.65	-0.891006110022918\\
			1.7	-0.951056074089937\\
			1.75	-0.987687880155987\\
			1.8	-0.999999534125777\\
			1.85	-0.987687880662437\\
			1.9	-0.951056072919503\\
			1.95	-0.891006109694265\\
			2	-0.809016617994674\\
			2.05	-0.707106452271519\\
			2.1	-0.587784978871094\\
			2.15	-0.453990288000256\\
			2.25	-0.156434392021627\\
			2.4	0.309016850631107\\
			2.45	0.453990288520828\\
			2.5	0.587784979121428\\
			2.55	0.707106451512455\\
			2.6	0.809016617394471\\
			2.65	0.891006109646179\\
			2.7	0.951056073954726\\
			2.75	0.987687880819145\\
			2.8	0.999999535130739\\
			2.85	0.98768788006749\\
			2.9	0.951056074238283\\
			2.95	0.891006108799756\\
			3	0.809016617464085\\
			3.05	0.707106452272733\\
			3.1	0.58778497835878\\
			3.15	0.453990288366192\\
			3.25	0.156434391974644\\
			3.4	-0.309016850563212\\
			3.45	-0.453990288356193\\
			3.5	-0.587784979090898\\
			3.55	-0.707106452273621\\
			3.6	-0.80901661804902\\
			3.65	-0.891006109724582\\
			3.7	-0.951056072453602\\
			3.75	-0.987687880993255\\
			3.8	-0.999999534115629\\
			3.85	-0.987687881415065\\
			3.9	-0.951056073948337\\
			3.95	-0.891006108953361\\
			4	-0.809016618239628\\
			4.05	-0.70710645193887\\
			4.1	-0.587784978501473\\
			4.15	-0.453990287026913\\
			4.25	-0.15643439187685\\
			4.4	0.309016850600848\\
			4.45	0.453990288105526\\
			4.5	0.587784978345372\\
			4.55	0.707106452448482\\
			4.6	0.809016618129526\\
			4.65	0.891006109900628\\
			4.7	0.951056072219137\\
			4.75	0.987687880309316\\
			6.75	2.98768693945105\\
			6.8	0.99999954459954\\
			7.75	0.999999530045235\\
			7.8	5.07690423035001e-09\\
			14.2	0\\
		};
		\addlegendentry{MPC}
		
		\addplot [color=orange, dotted, very thick]
		table[row sep=crcr]{%
			0	0.00243312504747273\\
			0.0500000000000007	-0.0313748739567146\\
			0.15	-0.00520753372268068\\
			0.199999999999999	-0.0166224096566232\\
			0.25	-0.0197909021542948\\
			0.300000000000001	-0.00428982396397792\\
			0.35	0.133868820891403\\
			0.4	0.258384813502742\\
			0.449999999999999	0.46183833153481\\
			0.5	0.555383281587227\\
			0.550000000000001	0.730313942922068\\
			0.6	0.791044094501869\\
			0.65	0.864450015057148\\
			0.699999999999999	0.964757724425029\\
			0.75	1.00377084452387\\
			0.800000000000001	0.959620952877565\\
			0.85	0.940684393706455\\
			0.9	0.910268744042448\\
			0.949999999999999	0.892250560365818\\
			1	0.830538841678987\\
			1.05	0.699050827879965\\
			1.1	0.549443820257567\\
			1.15	0.467101867470481\\
			1.2	0.341592989354494\\
			1.25	0.127343235272258\\
			1.3	0.0326642053386816\\
			1.35	-0.182455499251985\\
			1.4	-0.134775506682221\\
			1.45	-0.146685666198634\\
			1.5	-0.207754952956428\\
			1.55	-0.264055301129517\\
			1.6	-0.304715735156721\\
			1.65	-0.489475455894844\\
			1.7	-0.583860128175241\\
			1.75	-1.51314846640953\\
			1.8	-1.34876052491996\\
			1.85	-1.12359872827229\\
			1.9	-1.0279749937332\\
			1.95	-0.948368241162754\\
			2.05	-0.667859347850369\\
			2.1	-0.590242454713442\\
			2.15	-0.408264473842051\\
			2.2	-0.31017934133766\\
			2.25	-0.184065303938008\\
			2.4	0.339307157727399\\
			2.45	0.482104817109811\\
			2.5	0.598759318308284\\
			2.6	0.799879504614676\\
			2.65	0.877907464301456\\
			2.7	0.996246268617323\\
			2.75	1.02009179725651\\
			2.8	1.00060641020406\\
			2.85	1.04026366495205\\
			2.9	0.958930610778486\\
			2.95	0.95269155419901\\
			3.05	0.740112315448124\\
			3.1	0.582045725972947\\
			3.15	0.490529397837772\\
			3.2	0.316106790436896\\
			3.25	0.154352381621965\\
			3.35	-0.214525875432861\\
			3.4	-0.402881368030011\\
			3.45	-0.531895164626977\\
			3.5	-0.56446901438785\\
			3.55	-0.748897658697047\\
			3.6	-0.873634068738687\\
			3.65	-0.954031014468677\\
			3.7	-0.961036667123588\\
			3.75	-1.00896171914272\\
			3.8	-0.997954769253385\\
			3.85	-0.971266583502519\\
			3.9	-0.928147586807937\\
			3.95	-0.866830612567641\\
			4	-0.810538382105239\\
			4.05	-0.717718536711178\\
			4.1	-0.584287132780915\\
			4.15	-0.506318176155821\\
			4.2	-0.397519114147167\\
			4.3	-0.034972982276555\\
			4.35	0.096528734194381\\
			4.45	0.438334519840915\\
			4.5	0.602691488806689\\
			4.55	0.731160199547386\\
			4.6	0.850628576528917\\
			4.7	1.07115839155481\\
			4.75	1.11247649528787\\
			4.8	1.18620399764489\\
			4.85	1.16029341710518\\
			4.9	1.12738612102749\\
			4.95	1.12783207149539\\
			5	1.05632001755619\\
			5.05	0.995627175842843\\
			5.1	0.859744453511064\\
			5.15	0.793366316871548\\
			5.2	0.652326246675562\\
			5.25	0.64946628410544\\
			5.3	0.618918509370014\\
			5.4	0.684747797545674\\
			5.45	0.778222455370539\\
			5.5	0.974674354507659\\
			5.55	1.18438773484698\\
			5.6	1.43676853122441\\
			5.7	1.83944520887482\\
			5.75	1.95351904136281\\
			5.8	2.13016271304285\\
			5.85	2.25394633768917\\
			5.9	2.34577506848253\\
			6	2.31048688055861\\
			6.05	2.2482134899553\\
			6.1	2.26377131819867\\
			6.15	2.32547335926216\\
			6.2	2.30672132309364\\
			6.25	2.32271371932188\\
			6.3	2.32464332895852\\
			6.35	2.40092291819649\\
			6.4	2.496647643421\\
			6.45	2.66866339663255\\
			6.5	2.7239936650185\\
			6.55	2.81811255552152\\
			6.6	2.8565923816824\\
			6.65	2.79519980823668\\
			6.7	2.66015655768382\\
			6.75	2.38789108941468\\
			6.85	1.44932356709913\\
			6.95	0.970068182554797\\
			7	0.819623353512155\\
			7.05	0.772938265751888\\
			7.1	0.796776131167867\\
			7.15	0.784118001415548\\
			7.2	0.944107374472635\\
			7.25	1.0162506304014\\
			7.3	1.06556641074428\\
			7.35	1.05831465383154\\
			7.4	1.06793727012085\\
			7.45	1.02860982512018\\
			7.5	1.04265737467664\\
			7.55	1.02301406457089\\
			7.6	0.937829092077028\\
			7.65	0.890377457293948\\
			7.7	0.766252822072621\\
			7.75	0.566588577262698\\
			7.8	0.322291984127345\\
			7.85	0.195594274540809\\
			7.9	0.053033702336732\\
			7.95	-0.0444006702710382\\
			8	-0.0625456658002239\\
			8.05	-0.0605425814868319\\
			8.1	-0.100585472024322\\
			8.15	-0.0632181801188718\\
			8.2	-0.0195137550755486\\
			8.25	-0.0139717514099491\\
			8.3	0.0339833233518583\\
			8.35	0.0104149606992401\\
			8.4	-0.0259514267055163\\
			8.45	-0.067417472224184\\
			8.5	-0.044941345778005\\
			8.55	-0.0527710691227625\\
			8.6	-0.0420005180296634\\
			8.65	-0.0161012421257407\\
			8.7	-0.0343977613307835\\
			8.75	-0.0191612019842964\\
			8.8	-0.027086307364165\\
			8.85	-0.0233477668640081\\
			8.9	-0.0905962084883143\\
			8.95	-0.104376136654521\\
			9	-0.131829469809816\\
			9.05	-0.107490671103999\\
			9.15	-0.0756626025868794\\
			9.2	-0.0443310104838766\\
			9.25	-0.0267198176485728\\
			9.3	-0.0421741936231115\\
			9.35	-0.063741303533611\\
			9.4	-0.0422040532016883\\
			9.45	-0.0266057380308666\\
			9.5	-0.0183285516687182\\
			9.6	0.0616044960235982\\
			9.65	0.00798498051527297\\
			9.7	-0.0296276078933744\\
			9.8	-0.024670011082371\\
			9.85	-0.0568244939835996\\
			9.9	-0.066105223182241\\
			9.95	-0.0413018591008836\\
			10	-0.0197869126851522\\
			10.05	-0.0272429618873726\\
			10.1	-0.0067272368243021\\
			10.15	0.00277648238879991\\
			10.2	-0.0166078831494971\\
			10.25	-0.0162464257567176\\
			10.3	-0.00598687597666725\\
			10.35	-0.0371439159044442\\
			10.4	-0.101787919983359\\
			10.45	-0.0989097440534881\\
			10.5	-0.0743642854120541\\
			10.55	0.00829363714593079\\
			10.6	0.00787094015000633\\
			10.7	-0.00615969108713799\\
			10.75	-0.0281208907775898\\
			10.8	-0.0143970948854903\\
			10.85	0.0337875134136514\\
			10.9	0.0950687712086502\\
			10.95	0.0672924097203342\\
			11	0.019490214170375\\
			11.1	-0.0512233912782314\\
			11.15	-0.062019003333722\\
			11.2	-0.0879439521628758\\
			11.25	-0.0524362513066468\\
			11.3	-0.0546572887080483\\
			11.35	-0.0837231315624951\\
			11.4	-0.0969770661974181\\
			11.45	-0.0500351165585879\\
			11.5	0.0162751548708826\\
			11.55	0.00809264963813838\\
			11.6	0.0168351334908419\\
			11.65	-0.00428428454731922\\
			11.7	-0.0159227292791098\\
			11.75	-0.00597999279755612\\
			11.8	-0.0823505929498367\\
			11.85	0.131635330210704\\
			11.9	0.284216119955627\\
			11.95	0.341286976280191\\
			12	0.415317194153108\\
			12.05	0.495265268542289\\
			12.1	0.506742875846589\\
			12.15	0.582762367187094\\
			12.2	0.918427440061851\\
			12.25	1.56334951550541\\
			12.3	2.11236832582586\\
			12.35	2.12418990058559\\
			12.4	1.64441705330985\\
			12.45	1.10244065204988\\
			12.5	0.886033707171165\\
			12.55	0.771959195011911\\
			12.6	0.54045312036687\\
			12.65	0.202764000853939\\
			12.7	-0.0279130944221055\\
			12.75	-0.0460500760321256\\
			12.8	0.139751437310988\\
			12.85	0.431109026655658\\
			12.9	0.49706696084643\\
			12.95	0.373270106447462\\
			13	0.369561638832828\\
			13.05	0.44146525913173\\
			13.1	0.580562433389208\\
			13.15	0.659806832710688\\
			13.2	0.717144753807766\\
			13.25	0.660812354688296\\
			13.3	0.499692673748742\\
			13.4	-0.149585462434876\\
			13.45	-0.222531683654992\\
			13.5	-0.217661384187975\\
			13.55	-0.0537932352996364\\
			13.6	0.0461185051280264\\
			13.65	0.054670205149197\\
			13.7	-0.0136410152317463\\
			13.75	0.043183175540344\\
			13.8	0.179849345893185\\
			13.85	0.240534414240106\\
			13.9	0.237116136504698\\
			13.95	0.126941161609942\\
			14	0.129252313056252\\
			14.05	0.0916968095136497\\
			14.1	0.083190796454538\\
			14.15	0.0994009256129456\\
			14.2	0.00148423210369764\\
		};
		\addlegendentry{ODeePC}
		
		\addplot [color=cyan, dashed, very thick]
		table[row sep=crcr]{%
			0	-0.000469550622907988\\
			0.300000000000001	-1.29721913850744e-05\\
			0.4	0.308398011108396\\
			0.449999999999999	0.453084071982495\\
			0.5	0.587502625870004\\
			0.550000000000001	0.706960027229457\\
			0.6	0.808482053743393\\
			0.65	0.889916145384101\\
			0.699999999999999	0.950107934630733\\
			0.75	0.987265223303281\\
			0.800000000000001	0.999080576199205\\
			0.85	0.986484039600798\\
			0.9	0.949900575461371\\
			0.949999999999999	0.889892944191393\\
			1	0.807943380446073\\
			1.05	0.706075375394979\\
			1.1	0.58680959108948\\
			1.15	0.453132573996239\\
			1.25	0.15582215290533\\
			1.35	-0.156771237489307\\
			1.4	-0.167360489907503\\
			1.45	-0.182722477042116\\
			1.5	-0.213452906512799\\
			1.55	-0.266350046121406\\
			1.6	-0.344976112918005\\
			1.65	-0.455004365920741\\
			1.7	-0.589500100770842\\
			1.75	-2.59822193996438\\
			1.8	-1.9528845445964\\
			1.85	-1.81487113727843\\
			1.9	-0.50846823575014\\
			2	-0.197545783345099\\
			2.05	0.470994051838705\\
			2.1	0.304629596433315\\
			2.15	0.1041954384983\\
			2.2	-0.144763582124764\\
			2.25	-0.468355595930248\\
			2.3	-0.28829311143515\\
			2.35	-0.0576805363606994\\
			2.4	0.14238594553993\\
			2.45	0.280676645967912\\
			2.5	0.406437672253684\\
			2.55	0.652209736343908\\
			2.6	0.924853755091421\\
			2.65	1.0010105598143\\
			2.7	1.02080685574478\\
			2.75	1.05116391383766\\
			2.8	1.06947459794765\\
			2.85	1.0495434642551\\
			2.9	1.01238888476036\\
			2.95	0.954080091277644\\
			3	0.757683443486634\\
			3.15	0.246167514401961\\
			3.2	0.0893362800435344\\
			3.25	-0.05656375085481\\
			3.3	-0.190936488210815\\
			3.35	-0.311983289372217\\
			3.4	-0.423742058333429\\
			3.45	-0.499392235122441\\
			3.5	-0.583983232324284\\
			3.55	-0.693319188917073\\
			3.6	-0.785692249674316\\
			3.65	-0.857935848230717\\
			3.7	-0.911052331228481\\
			3.75	-0.947200194098752\\
			3.8	-0.966484621336463\\
			3.85	-0.966104399462974\\
			3.9	-0.942459936453174\\
			3.95	-0.890256670496138\\
			4	-0.769736899845411\\
			4.05	-0.664254704188409\\
			4.1	-0.540827856274426\\
			4.15	-0.404704429577341\\
			4.25	-0.109039652567885\\
			4.4	0.346156932916823\\
			4.45	0.489113824418657\\
			4.5	0.622583720709315\\
			4.55	0.743276661327529\\
			4.6	0.799192819907091\\
			4.65	0.878832360689456\\
			4.7	0.937181902047618\\
			4.75	0.971945317288654\\
			5	1.2099818670581\\
			5.1	1.33139206722127\\
			5.25	1.46949028942233\\
			5.35	1.56891977481691\\
			5.4	1.59960234707871\\
			5.45	1.64311866411849\\
			5.65	1.84036785504604\\
			5.75	1.94821539762129\\
			5.85	2.06420761953715\\
			6	2.2408434466848\\
			6.05	2.29247381173029\\
			6.1	2.29729559668707\\
			6.2	2.38922784630816\\
			6.3	2.49087749936358\\
			6.4	2.59838841916295\\
			6.45	2.66456574174283\\
			6.5	2.75553570930102\\
			6.7	2.95807245209536\\
			6.75	3.00973161330287\\
			6.8	1.02431508798856\\
			6.95	1.01718098218987\\
			7	1.01773111154459\\
			7.1	1.00890090713126\\
			7.15	0.971695635297884\\
			7.3	0.977870498016042\\
			7.4	0.978746852728055\\
			7.45	0.981147204423177\\
			7.5	0.996844846036588\\
			7.55	1.01596855403858\\
			7.7	1.01663043472857\\
			7.75	1.01692734541581\\
			7.8	0.0175366245680806\\
			8.05	0.0128033723384675\\
			8.1	-0.0195544160856489\\
			8.45	-0.0165194300778939\\
			8.5	-0.00648350905305506\\
			8.55	0.0175297700716488\\
			8.75	0.0170074929134607\\
			8.95	0.0176620073833345\\
			9	-0.0120851257472747\\
			9.05	-0.0178184988574017\\
			9.45	-0.018249391007739\\
			9.5	0.0126847055522923\\
			9.55	0.0182027834677072\\
			9.95	0.0186053412457809\\
			10	-0.0131051446251149\\
			10.05	-0.0185318619043855\\
			10.45	-0.0189185786867867\\
			10.5	0.01352991743126\\
			10.55	0.0188195505141433\\
			10.95	0.0191932536278507\\
			11	-0.013898088037001\\
			11.05	-0.019073299747502\\
			11.45	-0.0194375034870635\\
			11.5	0.0142225258464794\\
			11.55	0.0192873541434846\\
			11.75	0.019567500892526\\
			11.8	-0.122279980386592\\
			11.85	0.0181794343500172\\
			11.9	0.34121942736752\\
			11.95	0.334356906469853\\
			12.05	-0.07458030577763\\
			12.1	-0.177268650425916\\
			12.15	-0.198893436011959\\
			12.2	-0.204892572363757\\
			12.25	-0.134892761192775\\
			12.3	0.0224024585549998\\
			12.35	0.167422900539881\\
			12.4	0.15440374297801\\
			12.45	0.0963564481644976\\
			12.5	0.00892032302399137\\
			12.55	-0.0194067306785453\\
			12.6	-0.00413380710901556\\
			12.65	0.0829948203932283\\
			12.7	0.133892854673693\\
			12.75	0.0933938997619759\\
			12.8	-0.00340865591277506\\
			12.85	0.0151752461453736\\
			12.9	-0.0147351550214161\\
			12.95	-0.0150325003573091\\
			13	-0.00881194374083094\\
			13.05	0.0197734941630454\\
			13.15	0.0177531448178776\\
			13.25	0.0168281351310053\\
			13.3	0.0098108317040797\\
			13.35	-0.0217258438649974\\
			13.45	-0.0128230492857124\\
			13.5	-0.00543601174338093\\
			13.55	0.0106469047832594\\
			13.6	0.00935221359318028\\
			13.65	0.000832884120915978\\
			13.7	-0.018598380026754\\
			13.75	-0.0190557939687626\\
			13.8	-0.0160665968320952\\
			13.85	-0.00159226816290747\\
			13.9	-0.00414024431023741\\
			13.95	0.00814016769450809\\
			14	0.00400998293468646\\
			14.05	-0.0125473166754233\\
			14.15	-0.0146984492223527\\
			14.2	-0.0139505781462734\\
		};
		\addlegendentry{MDeePC}
		
		\addplot [color=magenta, dashdotted, very thick]
		table[row sep=crcr]{%
			0	0\\
			0.300000000000001	-1.29149185568167e-07\\
			0.4	0.309010336898464\\
			0.449999999999999	0.453981007303543\\
			0.5	0.587781965470461\\
			0.550000000000001	0.707104918428088\\
			0.6	0.80901110538149\\
			0.65	0.89099506105558\\
			0.699999999999999	0.951046538975483\\
			0.75	0.9876835031941\\
			0.800000000000001	0.999990210480213\\
			0.85	0.987675730734294\\
			0.9	0.951044358478965\\
			0.949999999999999	0.890994531237808\\
			1	0.809005735038729\\
			1.05	0.707095914080753\\
			1.1	0.587775028314065\\
			1.15	0.453981011364545\\
			1.25	0.156426608077899\\
			1.35	-0.156440183874329\\
			1.4	-0.167159798583834\\
			1.45	-0.182650701707125\\
			1.5	-0.213507649369074\\
			1.55	-0.26652815313596\\
			1.6	-0.345264128285306\\
			1.65	-0.449474532519845\\
			1.7	-0.574965284953924\\
			1.75	-0.713851954351195\\
			1.8	-0.860656224555543\\
			1.85	-0.986151840135305\\
			1.9	-0.95099716133975\\
			1.95	-0.890967031851806\\
			2	-0.808987966008399\\
			2.05	-0.70708438888923\\
			2.1	-0.587767704143021\\
			2.15	-0.453976759593752\\
			2.25	-0.156426662208732\\
			2.4	0.30901809119413\\
			2.45	0.45399020486118\\
			2.5	0.587783933051046\\
			2.55	0.707104515122944\\
			2.6	0.809013655780557\\
			2.65	0.891001976279538\\
			2.7	0.951050577797425\\
			2.75	0.987680915993764\\
			2.8	0.999991185728273\\
			2.85	0.987678737034919\\
			2.9	0.951046667837556\\
			2.95	0.89099673235434\\
			3	0.809007200957696\\
			3.05	0.707097144094767\\
			3.1	0.587775794858167\\
			3.15	0.453981118377174\\
			3.25	0.156425804970718\\
			3.4	-0.309022746951259\\
			3.45	-0.453994676810327\\
			3.5	-0.587787646067058\\
			3.55	-0.707107221372734\\
			3.6	-0.80901545825129\\
			3.65	-0.891003111198813\\
			3.7	-0.951051373984436\\
			3.75	-0.98768168355882\\
			3.8	-0.999992112128528\\
			3.85	-0.987679489090642\\
			3.9	-0.951047014233875\\
			3.95	-0.89099672323675\\
			4	-0.809007225828491\\
			4.05	-0.707097346824151\\
			4.1	-0.587776486379632\\
			4.15	-0.453982480624243\\
			4.25	-0.156428333668321\\
			4.4	0.309020480487488\\
			4.45	0.453993200078321\\
			4.5	0.587787105248681\\
			4.55	0.707107653242616\\
			4.6	0.809016613140187\\
			4.65	0.891004415757479\\
			4.7	0.951052151844744\\
			4.75	0.987681581290106\\
			6.75	2.98767675778653\\
			6.8	0.99999131750444\\
			7.75	0.999991096837995\\
			7.8	-6.96664003285719e-06\\
			11.75	5.96556053800157e-06\\
			11.8	-0.141857124155292\\
			11.85	-0.00150296308860476\\
			11.9	0.0243007044075405\\
			11.95	8.00642663989493e-05\\
			14.2	8.60499849153484e-08\\
		};
		\addlegendentry{MFDOOM}
		
		\addplot [color=black, forget plot]
		table[row sep=crcr]{%
			1.25	-3\\
			1.25	4\\
		};
		\addplot [color=white!15!black, forget plot]
		table[row sep=crcr]{%
			11.75	-3\\
			11.75	4\\
		};
		
		\addplot[area legend, draw=none, fill=white!90!black, fill opacity=0.3, forget plot]
		table[row sep=crcr] {%
			x	y\\
			1.25	-3\\
			11.75	-3\\
			11.75	4\\
			1.25	4\\
		}--cycle;
	\end{axis}
\end{tikzpicture}%
\caption{Simulated output response of an undisturbed \cref{eqn:sim} controlled by MPC, compared to ODeePC, MDeePC and MFDOOM controlling the same system under disturbance (shaded region).}
\label{fig:output_response}
\end{figure}
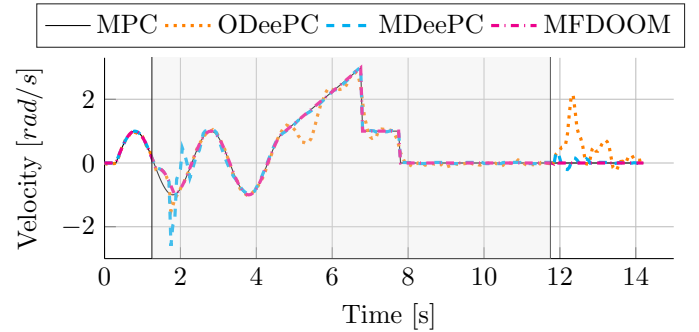

This paper introduces the Model-Free Disturbance Observer with Online Modification (MFDOOM), a model-free correction mechanism for autonomous LTV disturbances (ALTVDs). 
MFDOOM compares DeePC predictions with measured outputs and maintains a continuously updated Hankel matrix of prediction errors, which implicitly models the error dynamics for future compensation. 
Unlike methods that update the nominal behavior matrices, MFDOOM keeps the baseline DeePC predictor fixed and adapts only the correction mechanism. 
As illustrated in \Cref{fig:output_response}, this enables improved tracking under sinusoidal disturbances while requiring less real-time data updating than existing online DeePC variants.

The remainder of the paper is organized as follows. 
\Cref{sec:prelim} introduces preliminaries, \Cref{sec:problem} states the problem and elaborates on the limitations of existing DeePC variants, \Cref{sec:MFDOOM} introduces MFDOOM, and \Cref{sect: simulation} demonstrates its performance on a DC-motor example before \Cref{sec:concl} concludes the paper.

\section{PRELIMINARIES}\label{sec:prelim}
\subsection{Notations}
We denote the matrix of stacked $x_1, \dots, x_T\in \mathbb{R}^{n\times m}$ by
            \begin{equation*}
                \text{col}(x_1,\dots,x_T):=\begin{bmatrix}
                x_1\\ \vdots \\ x_T
            \end{bmatrix} \in \mathbb{R}^{nT \times m}
            \end{equation*}
and the order $L$ \emph{Hankel matrix} corresponding to $x := \text{col}(x_1,\dots,x_T)$ by \begin{equation*}
\mathcal{H}_L(x):=\begin{bmatrix}
            x_1&x_2&\dots&x_{T-L+1}\\
            x_2&x_3&\dots&x_{T-L+2}\\
            \vdots&\vdots&\dots&\vdots\\
            x_L&x_{L+1}&\dots&x_{T}\\
        \end{bmatrix}.
    \end{equation*}  
$\text{shift}(x,u):= \text{col}(x_2,\dots, x_T,u)$ stands for the \emph{shift operator} acting on $x$ and $u \in \mathbb{R}^{n\times m}$ and $\text{colspan}(M) \subset \mathbb{R}^{n}$ for the range of $M \in \mathbb{R}^{n \times m}$. Variables of optimization problems are stylized in bold (e.g., $\mathbf{g}$), and their optimal solution is denoted by an asterisk (e.g., $g^*$).

\subsection{Behavioral System Theory}
\emph{Behavioral system theory} describes a system as the subspace of the signal space in which trajectories of the system lie. A 'complete' linear time-invariant (LTI) system
\begin{equation}\label{eqn:LTI}
    \begin{aligned}
\bar{x}_{k+1}&=A\bar{x}_k+Bu_k, \; A\in\mathbb{R}^{n\times n},\; u_k\in\mathbb{R}^m,\; B\in\mathbb{R}^{n\times m} \\
    \bar{y}_k&=C\bar{x}_k+Du_k, \; C\in\mathbb{R}^{p\times n},\; D\in \mathbb{R}^{p\times m} 
    \end{aligned}
\end{equation}
has a signal space $\mathbb{W}\equiv\mathbb{R}^{m+p}$. Following \cite{DeePC}, we will denote the system by its behavior $\mathbb{B}$ and use $\mathbb{B}_T$ for the space of all its trajectories of length $T$. The order of an LTI system's minimal representation is denoted by $\mathbf{n}(\mathbb{B})$. Further,  the observability matrix of \cref{eqn:LTI} is given by $\mathcal{O}_l := \text{col}(C,CA,\dots,C^{l-1})\in\mathbb{R}^{lp\times n}$ and the system's lag $\mathbf{l}(\mathbb{B})$ is defined by $$\mathbf{l}(\mathbb{B}):=\min \{l\in\mathbb{N}:rank(\mathcal{O}_l)=\mathbf{n}(\mathbb{B})\}.$$
Letting the collected input/output data from system $\mathbb{B}$ be denoted by \begin{align*}
u_{ini}&:=\text{col}(u_1,u_2,\dots,u_{T_{ini}}),\; u_k\in\mathbb{R}^{m},\\
y_{ini}&:=\text{col}(y_1,y_2,\dots,y_{T_{ini}}),\; y_k\in\mathbb{R}^{p},
\end{align*}
it follows that for a sufficiently long window $T_{ini}>\mathbf{l}(\mathbb{B})$, the state to which the system is driven by the sequence of inputs $u_{ini}$ is unique \citep{DeePC}. The signal $u=\text{col}(u_1,u_2,\dots,u_T)$, $u_k\in\mathbb{R}^{m}$ is said to be persistently exciting (PE) of order $L$ if $\mathcal{H}_L(u)$ is of full row rank \citep{Persistency}. By the so-called \emph{Willems' Fundamental Lemma} \citet[Theorem~1]{Persistency}, every trajectory of length $t$ in $\mathbb{B}_t$ can be written as a linear combination of the columns of $\mathcal{H}_t(w)$ (see \Cref{lemma: Willems}). 
\begin{lemma}\label{lemma: Willems}
Consider a controllable system $\mathbb{B}$. Let $T, t \in \mathbb{Z}_{>0}$, and $w = \text{col}(u, y) \in \mathbb{B}_T$. Assume $u$ is persistently exciting of order $L=t+\mathbf{n}(\mathbb{B})$. Then $\text{colspan}(\mathcal{H}_t(w)) = \mathbb{B}_t$.
\end{lemma}
In particular, it is possible to link recorded trajectory data to feasible input-output pairs $(u,y)$. To see this, let $T$ be the total number of data points recorded and $T_{ini}\geq \mathbf{l}(\mathbb{B})$ and assume that $u$ and $y$ are of length $N$. To ensure that \Cref{lemma: Willems} holds for trajectories of length $T_{ini}+N$, we require persistency of excitation of order $T_{ini}+N+\mathbf{n}(\mathbb{B})$, which leads to the condition \begin{equation}\label{eqn:T size}
    T\geq (m+1)(T_{ini}+N+\mathbf{n}(\mathbb{B}))-1.
\end{equation}
Assume system $\mathbb{B}$ is driven by a PE input $u^d=\text{col}(u^d_1,u^d_2,\dots u^d_T)$, with corresponding output $y^d=\\\text{col}(y^d_1,y^d_2,\dots y^d_T)$, we can then construct from $u^d$ a Hankel matrix partitioned to the first $T_{ini}$ and last $N$ rows as follows:
\begin{equation}
    U:=\mathcal{H}_{(T_{ini}+N)}(u^d)=\begin{bmatrix}
        U_p^{T_{ini}\times (T-(T_{ini}+N)+1)}\\
        U_f^{N\times (T-(T_{ini}+N)+1)}\\
    \end{bmatrix}
\end{equation}
Similarly, we construct from $y^d$ a Hankel matrix:
\begin{equation}
    Y:=\mathcal{H}_{(T_{ini}+N)}(y^d)=\begin{bmatrix}
        Y_p^{T_{ini}\times (T-(T_{ini}+N)+1)}\\
        Y_f^{N\times (T-(T_{ini}+N)+1)}\\
    \end{bmatrix}
\end{equation}
If we separate a continuous trajectory into its first $T_{ini}$ points and last $N$ points, a trajectory belongs to $\mathbb{B}_{T_{ini}+N}$ if and only if $\exists g\in \mathbb{R}^{T-(T_{ini}+N)+1}$ such that
\begin{equation}\label{g reconstruct and predict}
    \begin{bmatrix}
        U_p\\Y_p\\U_f\\Y_f
    \end{bmatrix}g=\begin{bmatrix}
        u_{ini}\\y_{ini}\\u\\y
    \end{bmatrix}
\end{equation}
We will call matrices $U_p,U_f,Y_p,Y_f$ the \emph{behavior matrices}.

\subsection{Predictive Control}
\subsubsection{Model Predictive Control}\label{subsec: mpc}
Standard MPC \citep{MPC} is an optimization problem of the following form:
{
\begin{mini}
    {\mathbf{u,y,x}}{\sum_{k=0}^{N-1}{(\mathbf{y}_k-r_k)\transp Q(\mathbf{y}_k-r_k)+\mathbf{u}_k\transp R\mathbf{u}_k}}
    {\label{MPC original}}{}
    \addConstraint{\mathbf{x}_{k+1}}{=A\mathbf{x}_k+B\mathbf{u}_k}
    \addConstraint{\mathbf{y}_k}{=C\mathbf{x}_k+D\mathbf{u}_k.}
\end{mini}
}Given reference trajectory $\{r_k\}_{k=0}^{N-1},\quad r_k\in\mathbb{R}^p$ , initial state $x_0$ and known model $(A,B,C,D)$, MPC predicts output $\{\mathbf{y}_k\}_{k=0}^{N-1},\quad \mathbf{y}_k\in\mathbb{R}^p$ for any control input $\{\mathbf{u}_k\}_{k=0}^{N-1},\quad \mathbf{u}_k\in\mathbb{R}^m$. It therefore optimizes $\mathbf{u}_k$ and $\mathbf{y}_k$ such that the error between $\mathbf{y}_k$ and $r_k$ is minimized (measured by the weight matrix $Q$), and $\mathbf{u}_k$ is minimized (measured by the weight matrix $R$). Optimization is done over horizon length $N$, but usually only $u^*_0$ is used as the next control input (i.e. the "receding horizon" method).  It is also possible to add constraints on $\mathbf{u}, \mathbf{y}$ in \cref{MPC original}, but for simplicity we will discuss the unconstrained version. We denote the cost function in \cref{MPC original} as $J_{QR}(y,u,r)$.

\subsubsection{Data-Enabled Predictive Control} \label{subsec: deepc}
DeePC works conceptually like MPC, but considering the following optimization problem:
{
\begin{mini}
    {\mathbf{g,u,y},\boldsymbol{\sigma_y}}{J_{QR}(\mathbf{y,u},r)+K_g||\mathbf{g}||_2+K_{\sigma}||\boldsymbol{\sigma_y}||_2}
    {\label{DeePC original}}{}
    \addConstraint{\begin{bmatrix}U_p\\Y_p\\U_f\\Y_f\end{bmatrix}\mathbf{g}}
                 {=\begin{bmatrix}u_{ini}\\y_{ini}+\boldsymbol{\sigma_y}\\\mathbf{u}\\\mathbf{y}\end{bmatrix}}
\end{mini}
}
where $U_p,U_f,Y_p,Y_f$ are matrices generated offline, and $u_{ini},y_{ini}$ are the most recent $T_{ini}$ measurements of the input and output. In other words, the optimization performs the following:
\begin{itemize}
    \item Solves a feasibility problem for $\mathbf{g}$, as presented in \cref{g reconstruct and predict}, yielding model-free $\mathbf{u},\mathbf{y}$ prediction.
    \item Optimizes $\mathbf{u},\mathbf{y}$ based on weights $Q,R$.
\end{itemize}
The standard form of DeePC has $K_g,K_\sigma=0$, but \cite{DeePC} suggests to add $K_g,K_\sigma,\sigma_y$ to improve robustness. The DeePC (receding horizon) algorithm is detailed in \Cref{alg: DeePC}.
\begin{algorithm}[h]
\caption{Data-Enabled Predictive Control (DeePC).}\label{alg: DeePC}
\small
\begin{algorithmic}[1]
\Require $U_p,U_f$ were generated from a persistently exciting input. $Y_p,Y_f$ are the corresponding outputs.
\Require Hankel matrix dimensions are suitable for system complexity, as defined in \cite{DeePC}
\State Update $u_{ini},y_{ini}$ based on the last $T_{ini}$ measurements \label{DeePC first step}
\State Solve \Cref{DeePC original} for optimal $g^*, y^*,u^*$
\State Apply input $u^*(0)$ to the system
\State Return to \cref{DeePC first step}
\end{algorithmic}
\end{algorithm}

\section{PROBLEM STATEMENT}\label{sec:problem}
This paper focuses on the case of autonomous systems acting as input disturbances, i.e.:
\begin{equation}\label{eq: ALTI dist SS}
    \begin{aligned}
        w_{k+1}&=A_dw_k, & &w_k\in\mathbb{R}^{n_d},\; A_d\in\mathbb{R}^{n_d\times n_d}\\
        v_k&=C_dw_k, & &v_k\in\mathbb{R}^{m},\; C_d\in\mathbb{R}^{m\times n_d}
    \end{aligned}
\end{equation}
When $v_k$ is acting as an input disturbance into \cref{eqn:LTI}, the LTI acts as the disturbed model:
\begin{equation}\label{eqn:dist LTI}
    \begin{aligned} 
        \hat{x}_{k+1}&=\begin{bmatrix}
            A&BC_d\\0&A_d
        \end{bmatrix}\hat{x}_k+\begin{bmatrix}
            B\\0
        \end{bmatrix}u_k, \; \hat{x}_k:=\begin{bmatrix}
            \breve{x}\\ w
        \end{bmatrix}\in\mathbb{R}^{n+n_d}\\
        \hat{y}_k&=\begin{bmatrix}
            C&DC_d
        \end{bmatrix}\hat{x}_k+Du_k
    \end{aligned}
\end{equation}
Letting $\tilde{x}:=\bar{x}-\breve{x}$, the difference between the output of the undisturbed and the disturbed systems can be presented as the autonomous system $\mathbb{B}_d$:
\begin{equation}\label{eqn:dist LTI diff}
\begin{aligned}
    x_{k+1}&=\begin{bmatrix}
        A&-BC_d\\
        0&A_d
    \end{bmatrix}x_k, & & & x_k&:=\begin{bmatrix}
        \tilde{x}\\ w
    \end{bmatrix}\in\mathbb{R}^{n+n_d}\\
    d_k&=\begin{bmatrix}
        C&-DC_d
    \end{bmatrix}x_k & & & d&:=\bar{y}-\hat{y},
\end{aligned}
\end{equation}
Hence predictive control can handle ALTVDs in two ways: (i) by using the model in \cref{eqn:dist LTI} to predict $\hat{y}_k$ directly, or (ii) by using \cref{eqn:dist LTI diff} to predict $d_k$ in order to correct predictions made with the model of \cref{eqn:LTI} - which this paper employs.

\subsection{Handling LTV and Non-Linear Systems using DeePC}\label{sect: Handling LTV}
Before we present our solution, we will review how dynamics that are not represented by the originally generated behavior matrices are handled by the following three main methods: Regularized DeePC \citep{DeePC}, Online DeePC \citep{ODeePC} and Online Reduced-Order DeePC \citep{MDeePC}.

\subsubsection{Regularized DeePC}\label{subsec:reg deepc}
Regularization of $g$ and use of slack variables to handle noisy data and (to an extent) nonlinearities has been shown to improve robustness to model errors in practice and theory \citep{Regularization}.

\subsubsection{Online Data-Enabled Predictive Control}\label{subsec:odeepc}
As soon as data collection of $u^d, y^d$ is done, ODeePC keeps continuously updating the Hankel matrices in parallel to controlling the plant, as described in \Cref{alg: ODeePC}.

\begin{algorithm}[h]
\caption{Online Data-Enabled Predictive Control (ODeePC).}\label{alg: ODeePC}
{\small
\begin{algorithmic}[1]
\Require input $u$ always maintains persistency of excitation
\Require \Cref{alg: DeePC} requirements
\State collect $u^d,y^d$ and generate $U,Y$ while continuously updating $u_{ini},y_{ini}$\label{alg: OdeePC collect}
\State Immediately after \Cref{alg: OdeePC collect}, solve the DeePC algorithm.\label{alg: ODeePC start} 
\State inject $u^*(0)$ and measure the resulting $y$
\State update $u^d_{new} = \text{shift}(u^d,u^*(0))$ and $y^d_{new}= \text{shift}(y^d,y)$
\State update $u_{ini},y_{ini}$ as usual
\State update the Hankel matrices using the new $u^d,y^d$
\State return to \cref{alg: ODeePC start}
\end{algorithmic}
}
\end{algorithm}
Note that \cite{ODeePC} also provides an algorithm to make updates numerically efficient.

\subsubsection{Online Reduced-Order DeePC}\label{subsec:mdeepc}
Online Reduced-Order DeePC (MDeePC) employs a mosaic Hankel matrix, where new data is appended to the Hankel matrices, as opposed to ODeePC shifting it in \citep{MDeePC}. To keep the matrix order low and avoid adding uninformative data, \cite{MDeePC} evaluates how new data changes the rank of the behavior matrices via Singular Value Decomposition. If the smallest non-zero singular value of the updated matrix falls below a user-specified threshold, the new data are discarded. Similar to \cite{ODeePC}, a method for numerically efficient updates is also provided in \cite{MDeePC}.

\subsection{Limitations \& Drawbacks}
Each of the discussed methods comes with its own drawbacks. Regularization improves DeePC's robustness, but harms performance \citep{Regularization}. ODeePC reduces the need for heavy regularization, but is computationally more expensive and requires:
\begin{enumerate}
    \item  A PE input $u_k$ , which \cite{ODeePC} enforces by injecting random perturbations into $u^*(0)$.
    \item  Offline data $u_d,y_d$ that is continuous upon the initial condition $u_{ini},y_{ini}$, since otherwise discontinuities in the updated Hankel matrices introduce false dynamics. 
\end{enumerate}

These conditions have the following implications:
\begin{enumerate}[(i)]
    \item Injected perturbations drive the output away from the reference --- causing unavoidable errors.
    \item Without PE input, the collected data no longer represent the system dynamics. If old data is then shifted out, ODeePC risks “erasing” its own model.
    \item Matrix updating must always remain active, which is why noise must be injected continuously, even if the dynamics have not changed since data collection.
\end{enumerate}

MDeePC avoids many of these issues by using mosaic Hankel matrices. However, although \cite{MDeePC} provides an algorithm to make updates efficient, it is seen that the updated matrix still grows in size even when returning to previously explored system configurations. This suggests that the optimization becomes increasingly expensive in time.

\section{Model-Free Disturbance Observer with Online Modification}
\label{sec:MFDOOM}
 Model-Free Disturbance Observer with Online Modification (MFDOOM) aims to treat ALTVDs injected into LTIs (see \cref{eqn:dist LTI}). Examples include a step disturbance ($A_d,C_d$ and $w_0$ are 1), a ramp with
 $
\begin{aligned}
    A_d=\begin{bmatrix}
        1&1\\0&1
    \end{bmatrix},C_d=\begin{bmatrix}
        1&0
    \end{bmatrix},w_0=\begin{bmatrix}
        0&1
    \end{bmatrix}\transp
\end{aligned}
$,
and a sine disturbance of amplitude $a$, frequency $\Omega$ and phase $\Phi$, with step length $T_s$, given by
\begin{align*}
    A_d&=\begin{bmatrix}
        \cos(\theta)&\sin(\theta)\\-\sin(\theta)&\cos(\theta)
    \end{bmatrix}, \; C_d=\begin{bmatrix}
        1&0
    \end{bmatrix}, \;\theta =\Omega\cdot 2\pi \cdot T_s\\
    w_0&=\begin{bmatrix}
        a\cdot \cos(\Phi)&-a\cdot \sin(\Phi)
    \end{bmatrix}.
\end{align*}
These could also be used as building blocks for time-varying disturbances, e.g., a ramp disturbance with varying slope.
\citet[Corollary~21]{Identifiability} shows that, for all $L>\mathbf{l}(\mathbb{B})$,
\[
\text{colspan}\!\left(\mathcal{H}_L(\omega)\right)=\mathbb{B}_{L}
\]
if and only if $\text{rank}\left(\mathcal{H}(\omega)\right)=mL+n$.
The following proposition states the consequence used by MFDOOM.
\begin{proposition}
\label{prop:auto_error_prediction}
Let $d^d$ be a trajectory generated by an autonomous LTI behavior $\mathbb{B}_d$ such as \cref{eqn:dist LTI diff}, and define
\[
D :=
\begin{bmatrix}
D_p\\D_f
\end{bmatrix}
:=
\mathcal{H}_{T_{ini}+N}(d^d),
\]
where $D_p$ contains the first $T_{ini}$ block rows and $D_f$ the following $N$ block rows. 
If $T_{ini}+N>\mathbf{l}(\mathbb{B}_d)$ and
\[
\text{rank}\!\left(\mathcal{H}_{T_{ini}+N}(d^d)\right)
=
\mathbf{n}(\mathbb{B}_d),
\]
then there exists $g_d$ for every trajectory 
$\text{col}(d_{ini},d_{predict})\in\mathbb{B}_{d,T_{ini}+N}$ 
such that
\[
\begin{bmatrix}
D_p\\D_f
\end{bmatrix}g_d
=
\begin{bmatrix}
d_{ini}\\d_{predict}
\end{bmatrix}.
\]
In particular, once $D_p g_d=d_{ini}$ is solved for $g_d$, $D_f g_d$ provides the corresponding future prediction-error trajectory.
\end{proposition}
\begin{pf}
By \citet[Corollary~21]{Identifiability}, the stated rank condition with $m=0$ implies
$\text{colspan}(D)=\mathbb{B}_{d,T_{ini}+N}$.
Hence any admissible trajectory $\text{col}(d_{ini},d_{predict})$
of the autonomous prediction-error behavior is a linear combination of the columns of $D$,
which gives the stated result after partitioning $D=\text{col}(D_p,D_f)$.
\end{pf}

\subsection{MFDOOM Algorithm}
MFDOOM considers the following optimization problem:

\begin{mini}
    {\mathbf{g,u,y},\boldsymbol{\sigma_y},\mathbf{d,g_d}}
    {J_{QR}(\mathbf{y}-\Gamma \mathbf{d},\mathbf{u,r})
    +K_g||\mathbf{g}||_2
    +k_\sigma||\boldsymbol{\sigma_y}||_2}
    {\label{MFDOOM}}{}
    \breakObjective{+k_{g_d}||\mathbf{g_d}||_2}
    \addConstraint{\begin{bmatrix}U_p\\Y_p\\U_f\\Y_f\end{bmatrix}\mathbf{g}}
    {=\begin{bmatrix}u_{ini}\\y_{ini}+\boldsymbol{\sigma_y}\\\mathbf{u}\\\mathbf{y}\end{bmatrix}}
    \addConstraint{\begin{bmatrix}D_p\\D_f\end{bmatrix}\mathbf{g_d}}
    {=\begin{bmatrix}d_{ini}\\\mathbf{d}\end{bmatrix}}
\end{mini}

Variables $\mathbf{u,y,g},\boldsymbol{\sigma_y}$, along with $U_p,Y_p,U_f,Y_f,u_{ini}$ and $y_{ini}$, follow the standard DeePC formulation. 
MFDOOM augments DeePC with prediction-error behavior matrices $D_p,D_f$, continuously generated from collected prediction-error samples $d^d$. 
Recent prediction-error history $d_{ini}\in\mathbb{R}^{p\cdot T_{ini_d}}$, along with $\mathbf{g_d}$ enables prediction of future error $\mathbf{d}$, which corrects the predicted output in the cost through $\mathbf{y}-\Gamma\mathbf{d}$. 
The online implementation is summarized in \Cref{alg: MFDOOM}: data collection starts only after the prediction error exceeds user-defined numerical noise threshold $\epsilon$, correction is enabled once $D_p,D_f$ are generated, and $\Gamma$ may be adapted by \Cref{alg: Adaptive MFDOOM} to suppress unreliable corrections (i.e. for large error-estimation mismatch, $\Gamma d\approx0$, reverting \cref{MFDOOM} to regularized DeePC). Tuning $K_{g_d}$ and $K_\Gamma$ therefore balances disturbance rejection, transient robustness, and nominal performance.

\begin{algorithm}[h]
\caption{Model-Free Disturbance Observer with Online Modification (MFDOOM).}\label{alg: MFDOOM}
{\small
\begin{algorithmic}[1]
\Require $T_d \geq T_{ini_d}+N+\mathbf{n}(\mathbb{B}_d)-1$
\Require $\epsilon$ bigger than value of expected numerical errors
\Require All DeePC conditions for $U_p,Y_p,U_f,Y_f$ hold
\State Set $\Gamma$, $D_p$ and $D_f$ to zero
\State Solve \cref{MFDOOM} to get $u^*,y^*$ (and $d^*$) \label{alg: doom dpc start}
\State Apply $u^*(0)$ to the system
\State Measure the actual output $y_a$ 
\State Calculate the prediction error $d_a := y^*(0)-y_a$ \label{alg: doom dpc end}
\If {$|d_a|<\epsilon$}
    \State Return to \cref{alg: doom dpc start}
\Else
    \State Move to \cref{alg: doom while}
\EndIf
\While {size($d^d$)$<T_d$}\label{alg: doom while}
    \State Append $d_a$ to $d^d$
    \State Repeat \cref{alg: doom dpc start}-\cref{alg: doom dpc end}
\EndWhile 
\State Update $d^d_{new}=\text{shift}(d^d,d_a)$ \label{alg: doom shift}
\State Generate $D_p$ and $D_f$ using $\mathcal{H}(d^d)=\text{col}(D_p,D_f)$ \label{alg: doom start}
\State Set $\Gamma=1$
\State Repeat \cref{alg: doom dpc start}-\cref{alg: doom dpc end}
\State Return to \cref{alg: doom shift}
\end{algorithmic}
}
\end{algorithm}
 
\begin{algorithm}[h]
\caption{Adaptive Gain MFDOOM.}\label{alg: Adaptive MFDOOM}
{\small
\begin{algorithmic}[1]
\Require $K_\Gamma$ user determined adaptive gain
\State Perform \Cref{alg: MFDOOM} up to \cref{alg: doom start}
\State solve \cref{MFDOOM} to get $u^*,y^*,d^*$ \label{alg: adaptive doom start}
\State $\text{use $u^*(0)$ as the input to the system}$
\State Measure the actual output $y_a$ 
\State Calculate the actual prediction error $d_a := y^*(0)-y_a$
\State Calculate prediction-error estimation error $e_d=|d_a-d^*(0)|$
\State Set $\Gamma = \frac{1}{(1+K_\Gamma\cdot e_d^2)}$
\State shift out oldest $d^d$ values, shift in $d_a$ 
\State Generate $D_p,D_f$ using $d^d$
\State Return to \cref{alg: adaptive doom start}
\end{algorithmic}
}
\end{algorithm}

\subsection{Convergence}
\cite{DeePC} provides proof of the convergence of standard DeePC algorithm. As MFDOOM simply adds a correction to output prediction $\mathbf{y}$, convergence of MFDOOM is reliant on two aspects: is the $\mathbf{d}$ prediction accurate, and can $D_p,D_f$ be continuously updated. Prediction accuracy is given by \Cref{prop:auto_error_prediction}. Matrix update between iterations is not an issue as long as the conditions of \Cref{prop:auto_error_prediction} stay valid (equivalent to ODeePC). Since $d$ is not the actual difference between two systems, but rather the difference between the measured system output and the DeePC predictor, it may be noisy. Analogously to DeePC regularization, penalizing $\mathbf{g_d}$ is expected to improve robustness to noisy prediction-error data; this is verified numerically in \Cref{sect: simulation}.

\subsection{Benefits}

MFDOOM has several practical advantages over existing online DeePC variants. Since the nominal DeePC matrices $U,Y$ remain fixed, closed-form DeePC terms depending only on $U,Y$ can be precomputed \citep{Regularization}, while online updates are required only for the smaller prediction-error matrix $D:=\text{col}(D_p,D_f)$, which does not expand in size. Unlike ODeePC, MFDOOM does not require persistently exciting probing inputs for the nominal behavior; the required richness is shifted to $D$, i.e., to the observed prediction-error behavior. If the data are low-rank or unreliable, the learned correction is accordingly low-order, irrelevant, or temporarily unreliable, but inaccuracies are attenuated through $\Gamma$ -- greatly improving transient response. This separation between the fixed baseline DeePC predictor and the adaptive correction mechanism allows prediction-error data to be reset or model-correction to be attenuated without corrupting the nominal controller, unlike methods that continuously modify or expand the nominal behavior matrices.

\section{Simulations}\label{sect: simulation}
We control the velocity of a DC motor under ZOH discretization with sample period of $T_s=0.05[s]$ and sinusoidal input disturbance, described by the following state space matrices (shown to 6th digit after decimal point accuracy):
\begin{equation}\label{eqn:sim}
\begin{aligned}
    A &= 0.010615 & & & B &= 0.174131 \\
    C &= 26.363636 & &  & D &= 0 \\
    A_d &=\begin{bmatrix}
        \phantom{-}0.951056 &	0.309016\\
        -0.309016 &	0.951056
    \end{bmatrix} & & &
     C_d &= \begin{bmatrix}
        1 & 0
    \end{bmatrix} 
\end{aligned}
\end{equation}
We define the tracking error $e_k\in\mathbb{R}$ as the difference between the reference and the feedback at step $k$. For tracking error until step $T$, we measure its size via the root mean square error (RMSE) as defined by: \begin{equation*}
    \text{RMSE}:=\sqrt{\frac{1}{T}\sum_{k=1}^T{e_k^2}}.
\end{equation*}

We simulate four controllers --- MPC without disturbance as a performance benchmark, and ODeePC, MDeePC and MFDOOM under disturbance. Tuning parameters are detailed in \Cref{tab:tuning params}, and regularization terms are different for each method, as ODeePC and MDeePC lose stability for lower $K_g$ values. The reference trajectory is composed of sine, ramp and step segments (see~\Cref{fig:output_response}). The test is repeated with added sensor noise of amplitude $\pm 0.005 \;[rad/s]$ and adjusted MFDOOM gains $K_g=10,K_{g_d}=10$ (see~\Cref{fig:noisy_error_response}). All tests use the same noise vector, whose elements were generated with MATLAB's rand() function, i.e. $0.01\cdot\left(rand()-0.5\right)$.

\begin{table}[h]
\begin{center}
\caption{Simulation Tuning Parameters.}\label{tab:tuning params}
    \begin{tabular}{c|c|c|c|c}
    \hline
            & MPC    & ODeePC & MDeePC & MFDOOM\\
         \hline
        $T_{ini}$    & 20     & 20     & 20   & 20  \\
         \hline
        $N$ & 5      & 5      & 5      & 5  \\
        \hline
        $Q$ & 1000   & 1000   & 1000   & 1000  \\
         \hline
        $R$ & 0.1    & 0.1    & 0.1    & 0.1\\
         \hline
        $K_g$ & -    & 1000   & 100    & 1\\
         \hline
        $K_\sigma$&- & $10^5$ & $10^5$ & $10^5$\\
         \hline
        $T_{ini_d}$    & -     & -     & -   & 2  \\
         \hline
        $K_{g_d}$& - & -      & -      & 1\\
         \hline
        $K_\Gamma$& - & -      & -    & 1000\\

    \hline
    \end{tabular}
\end{center}
\end{table} 

Like \cite{MDeePC}, our example corresponds to an LTV system that is piecewise LTI, with the autonomous disturbance active in the shaded regions of the response plots. The output response is shown in \Cref{fig:output_response} and the error plot is shown in \Cref{fig:error_response}. Compared to ODeePC and MDeePC, MFDOOM exhibits faster convergence, close to zero tracking error after prediction convergence, smaller error peaks and lower RMSE. Further, \Cref{fig:disturbance_response} visually shows that MFDOOM $d^*$ converges to the actual prediction-error trajectory.

\begin{figure}[ht]
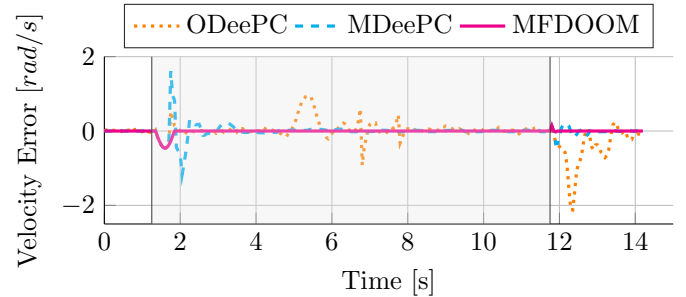
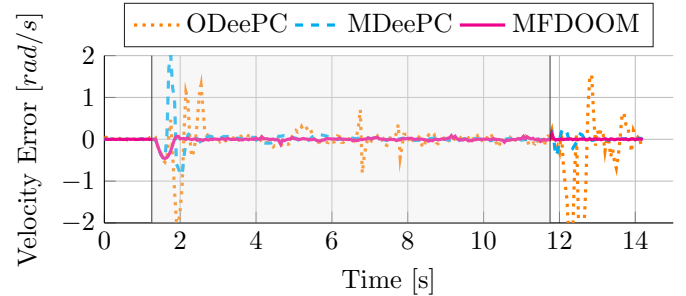

\centering

\begin{subfigure}{\linewidth}
    \centering
%
    \caption{RMSE -- ODeePC: 0.34, MDeePC: 0.19, MFDOOM: 0.06.}
    \label{fig:noiseless_error_response}
\end{subfigure}

\vspace{0.3cm}

\begin{subfigure}{\linewidth}
    \centering
   %
%
    \caption{RMSE -- ODeePC: 0.55, MDeePC: 0.22, MFDOOM: 0.06.}
    \label{fig:noisy_error_response}
\end{subfigure}

\caption{Velocity tracking error comparison of ODeePC, MDeePC and MFDOOM when controlling \cref{eqn:sim} under harmonic disturbance (gray region), without feedback sensor noise (\Cref{fig:noiseless_error_response}) and with it (\Cref{fig:noisy_error_response}).}
\label{fig:error_response}

\end{figure}

\begin{figure}[ht]
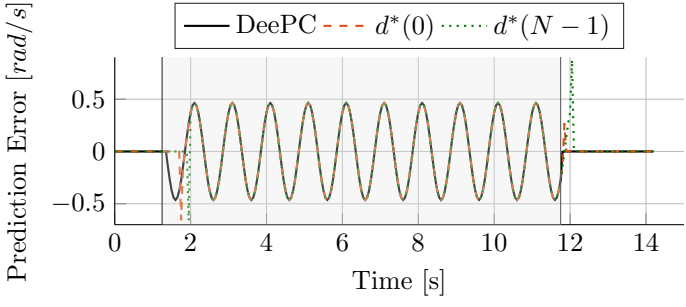
\centering
\definecolor{mycolor1}{rgb}{0.85000,0.32500,0.09800}%
\definecolor{mycolor2}{rgb}{0,0.5,0.0}%
%
%
\caption{DeePC output prediction-error due to a piecewise-LTI harmonic disturbance acting on \cref{eqn:sim} compared to the 0-step and $(N-1)$-step error-estimations found by MFDOOM.}
\label{fig:disturbance_response}
\end{figure}

\subsection{Time-Varying-Amplitude Linear Disturbance}
Next, we change system \cref{eqn:sim} by replacing $C_d$ with $\tilde{C}_d(t)=\frac{t}{100}C_d$. Again, we observe that MFDOOM both accurately predicts future prediction errors as well as correctly compensates for them (see~\Cref{fig:LTV_error_response,fig:LTV_disturbance_response}).
\begin{figure}[ht]\centering
\begin{tikzpicture}
	
	\begin{axis}[%
		width=0.85\columnwidth,   %
		height=0.25\columnwidth,  %
		scale only axis,
		unbounded coords=jump,
		xmin=0,
		xmax=15,
		ymin=-2,
		ymax=2,
		xlabel={Time [s]},
		ylabel={Velocity Error $[rad/s]$},
		axis background/.style={fill=white},
		axis x line*=bottom,
		axis y line*=left,
		xmajorgrids,
		ymajorgrids,
		legend style={
			at={(0.5,1.05)},
			anchor=south,
			legend columns=-1,
			legend cell align=left,
			align=left,
			draw=white!15!black,
		}
		]
		\addplot [color=orange, dotted, very thick]
		table[row sep=crcr]{%
			0	0.000174582125810119\\
			0.0500000000000007	0.0275185545792223\\
			0.0999999999999996	-0.0464911137051072\\
			0.15	-0.0275123358345919\\
			0.25	-0.0454176909385478\\
			0.300000000000001	-0.0150699690361549\\
			0.35	-0.0187741408501587\\
			0.4	-0.0104906591772007\\
			0.449999999999999	-0.0216361176976818\\
			0.5	-0.00761483272540708\\
			0.550000000000001	0.0362166654105671\\
			0.6	0.0364960480206999\\
			0.65	-0.0414219310581156\\
			0.699999999999999	0.00358554486857976\\
			0.75	-0.00739137045103178\\
			0.800000000000001	0.000511843009626034\\
			0.85	-0.0197229915709407\\
			0.9	0.0352463436129646\\
			1	-0.00127924925786083\\
			1.05	-0.0277381781253023\\
			1.1	-0.0282617723227609\\
			1.15	-0.0314426139114907\\
			1.2	-0.0222653007537712\\
			1.25	-0.0314718412165895\\
			1.3	0.0213607332918784\\
			1.35	-0.0402885522599501\\
			1.4	-0.18065833526744\\
			1.45	-0.391184747636792\\
			1.5	-0.553353573153355\\
			1.55	-0.583732555819651\\
			1.6	-0.69043423509329\\
			1.65	-0.58008903135396\\
			1.7	-0.517672556593443\\
			1.75	0.686113461276879\\
			1.8	0.531488440221132\\
			1.85	0.251114250848048\\
			1.9	0.0852105099361307\\
			1.95	0.0292264870908685\\
			2	-0.0186024207102271\\
			2.05	0.00845082068703107\\
			2.1	-0.0300940700884311\\
			2.15	0.0169651745628627\\
			2.2	0.000611458877401105\\
			2.25	0.012059931363094\\
			2.3	-0.0198246228232168\\
			2.35	0.0203509309494248\\
			2.4	-0.0166070940328229\\
			2.45	0.000356958347422776\\
			2.5	0.00252268246525134\\
			2.55	0.00631293168476788\\
			2.6	-0.00657435700601816\\
			2.65	0.00560147045895931\\
			2.7	-0.0516357680826385\\
			2.75	-0.0674644246803755\\
			2.8	-0.0749563335725405\\
			2.95	0.0137075982045847\\
			3	0.0162491618997898\\
			3.05	0.00900930323436988\\
			3.1	0.0147286741864701\\
			3.15	-0.0341836314692561\\
			3.2	0.0137677674807328\\
			3.25	0.0123813552802243\\
			3.3	0.0404293245240854\\
			3.35	0.075585404005853\\
			3.4	0.0819052050967439\\
			3.45	0.0285378909941265\\
			3.5	0.067081028108408\\
			3.55	0.0339236861123577\\
			3.6	0.0584522088833701\\
			3.65	0.025599252798207\\
			3.7	-0.0262311002062425\\
			3.75	0.0763714091368399\\
			3.8	0.0145510194133944\\
			3.85	0.0122024992316874\\
			3.9	-0.00297729800943358\\
			3.95	0.0272322953382265\\
			4	0.0800416579120871\\
			4.05	0.0838818240649211\\
			4.1	0.166441595789523\\
			4.15	0.190487136475037\\
			4.2	0.146353058269959\\
			4.25	0.200730829286417\\
			4.3	0.124819747949628\\
			4.35	0.115909464256969\\
			4.4	0.0987494729999323\\
			4.45	0.0744309760906958\\
			4.5	0.0905416144130715\\
			4.55	0.0303006277657598\\
			4.6	0.0281538870238656\\
			4.65	-0.008813781776702\\
			4.7	-0.043374905610019\\
			4.75	-0.106730734343129\\
			4.8	-0.0665321315351211\\
			4.85	-0.0106297714645365\\
			4.9	0.0391201981665983\\
			4.95	0.172359393258212\\
			5	0.243456232750418\\
			5.05	0.382190904127727\\
			5.1	0.516391399642451\\
			5.15	0.705739205253623\\
			5.2	0.841463659088026\\
			5.3	1.14634459560545\\
			5.35	1.30793785124174\\
			5.4	1.39493818624685\\
			5.45	1.5231588524037\\
			5.5	1.52805711749025\\
			5.55	1.47607187132311\\
			5.6	0.864648761760467\\
			5.65	0.702868775597132\\
			5.7	0.450405240447957\\
			5.75	0.273010617108682\\
			5.8	0.283861271423307\\
			5.85	-0.0586712948584847\\
			5.9	-0.011095691736891\\
			5.95	-0.0924515789165827\\
			6	-0.12484634700175\\
			6.05	0.0964675012523433\\
			6.1	-0.028424261047828\\
			6.15	0.0266636754635527\\
			6.2	0.0064901820841996\\
			6.25	-0.0436408856856616\\
			6.3	0.0867931808728777\\
			6.35	0.0120177082675905\\
			6.4	0.0702238329577654\\
			6.45	0.0692415001123692\\
			6.5	-0.0379047387377529\\
			6.55	0.0355493199151002\\
			6.6	0.0313570240982681\\
			6.65	0.0046378723702265\\
			6.7	0.167065187273108\\
			6.75	0.466918114855007\\
			6.8	-0.551740043509646\\
			6.85	-0.239869286021095\\
			6.9	-0.0535715077819745\\
			6.95	0.0395890978078803\\
			7	-0.0649157568497536\\
			7.05	0.0355250552651096\\
			7.1	-0.0566777430670573\\
			7.15	-0.00540229206196763\\
			7.2	0.124219427313275\\
			7.25	-0.0518311498605044\\
			7.3	0.0630276497792828\\
			7.35	0.0673528821815914\\
			7.4	0.0389772475283792\\
			7.45	0.0196733122657839\\
			7.5	-0.0281044339915208\\
			7.55	0.0625787962772595\\
			7.6	0.00501174526776182\\
			7.65	-0.0445836428447208\\
			7.7	0.0903914739906924\\
			7.75	0.128187130999573\\
			7.8	-0.223892174053587\\
			7.85	-0.115039503210589\\
			7.9	0.0075992926882229\\
			7.95	0.0339686006429076\\
			8	-0.0553159263126553\\
			8.05	0.121829703750375\\
			8.1	0.00525772594357043\\
			8.15	-0.0125749994084323\\
			8.2	0.0714172091641299\\
			8.25	-0.101899146509645\\
			8.3	0.00184814154270718\\
			8.35	0.0175307514751442\\
			8.4	-0.00842817088457615\\
			8.45	-0.0169954619275323\\
			8.5	-0.109457979599684\\
			8.55	-0.0358786678294578\\
			8.6	-0.0212174709099102\\
			8.65	-0.0265165972942469\\
			8.7	0.0565054693447156\\
			8.75	-0.00544077848176627\\
			8.8	0.0706887177403406\\
			8.85	0.0495378096222421\\
			8.9	0.0903777372900443\\
			8.95	0.0981372333224666\\
			9	-0.0246318553432356\\
			9.05	0.0712741489048074\\
			9.1	-0.0130904820176418\\
			9.15	-0.0352281203120697\\
			9.2	-0.0210184622441503\\
			9.25	-0.0340999779145985\\
			9.3	0.0627474277017548\\
			9.35	0.0341536218944505\\
			9.4	0.0270623940271069\\
			9.45	0.0862535380909808\\
			9.5	-0.0027401384702852\\
			9.55	-0.00604619436355414\\
			9.6	0.00306077240096592\\
			9.65	-0.0317619881673163\\
			9.7	-0.000404799420040902\\
			9.75	-0.0950523246735866\\
			9.8	-0.0508505148101719\\
			9.85	-0.030365038672489\\
			9.9	-0.031219156444017\\
			9.95	0.0203172911116294\\
			10	-0.0530549341101167\\
			10.05	0.0223261005394555\\
			10.1	0.0256578834056338\\
			10.15	-0.0289032682776131\\
			10.2	0.0309644281435997\\
			10.25	-0.0126891275474428\\
			10.3	0.0518843621990026\\
			10.35	0.00581882519684029\\
			10.4	-0.000989408223247779\\
			10.45	0.0699934035319885\\
			10.5	0.0345568813835548\\
			10.55	0.0642492148061553\\
			10.6	0.105322713716854\\
			10.65	0.0998725694870632\\
			10.7	0.109994528632637\\
			10.75	0.0721488538867767\\
			10.8	0.0840307198315386\\
			10.85	0.0829687520341125\\
			10.9	-0.00786420572897839\\
			10.95	0.0178127093692275\\
			11	-0.0255793491170468\\
			11.05	-0.00309662060276672\\
			11.1	-0.0355409688318034\\
			11.15	-0.0484703842490521\\
			11.2	0.0170751710899992\\
			11.25	-0.0630916871205045\\
			11.3	-0.0291181386597561\\
			11.35	-0.0175229123803273\\
			11.4	-0.0300293802572256\\
			11.45	0.0665861720033014\\
			11.5	0.00161327411982271\\
			11.55	0.0659693479821062\\
			11.6	0.110041907322627\\
			11.65	0.0988612082882554\\
			11.7	0.106577025727651\\
			11.75	0.0150706782282732\\
			11.8	0.463509252729303\\
			11.85	0.0112753883252843\\
			11.9	-0.543805984434879\\
			12	-1.30516034036796\\
			12.05	-1.53475734391911\\
			12.1	-1.70369398448737\\
			12.15	-2.32219744165318\\
			nan	nan\\
			12.461295609769	-2.4\\
			12.5098943488114	2.4\\
			nan	nan\\
			12.8753323778572	2.4\\
			12.9	0.837345435410104\\
			12.95	-1.57965576678189\\
			12.9983632396594	-2.4\\
			nan	nan\\
			13.0016281198775	-2.4\\
			13.05	-1.57515484581877\\
			13.1	0.553249119402436\\
			13.1472108258622	2.4\\
			nan	nan\\
			13.3502743908996	2.4\\
			13.4	0.0613935577193665\\
			13.45	-1.29486205750595\\
			13.5	-1.75546921843104\\
			13.5105553361627	-2.4\\
			nan	nan\\
			13.7991124883228	-2.4\\
			13.7991124883228	-2.4\\
			nan	nan\\
			13.6261030599573	-2.4\\
			13.65	-0.637562494113569\\
			13.7	1.98988732669885\\
			13.7058378107475	2.4\\
			nan	nan\\
			14.0633466549434	2.4\\
			14.093347610738	-2.4\\
			14.0812958494798	-2.4\\
			14.1935905685637	-2.4\\
			14.2	-1.82514838200069\\
		};
		\addlegendentry{ODeePC}
		
		\addplot [color=cyan, dashed, very thick]
		table[row sep=crcr]{%
			0	0.0303668471226111\\
			0.0500000000000007	0.0211068597368573\\
			0.0999999999999996	-0.0170193206553844\\
			0.15	-0.0305824659801175\\
			0.199999999999999	-0.00873402430223535\\
			0.25	0.00239852302342314\\
			0.300000000000001	-0.078863112995764\\
			0.35	-0.0368685363945147\\
			0.4	0.0170464195630977\\
			0.449999999999999	0.0174847867179189\\
			0.550000000000001	0.116198606694358\\
			0.6	0.0961986741290044\\
			0.65	0.00498202682378235\\
			0.699999999999999	-0.0202348863854311\\
			0.75	-0.0173251918448987\\
			0.800000000000001	0.0089511236039197\\
			0.85	0.0530111262816373\\
			0.9	0.059237927788832\\
			0.949999999999999	0.0519932999645718\\
			1	0.0514003630863975\\
			1.1	0.0387112295893228\\
			1.2	0.00839277558150364\\
			1.25	-0.00442772902686173\\
			1.3	-0.0155262434552395\\
			1.35	-0.0286192724903565\\
			1.4	-0.231729900937344\\
			1.45	-0.421314625145959\\
			1.5	-0.577222889691189\\
			1.55	-0.677929774235432\\
			1.6	-0.71398892751775\\
			1.65	-0.682725751960595\\
			1.7	-0.589394884781996\\
			1.75	-0.36852613662511\\
			1.8	0.541793232088338\\
			1.85	0.562143705567671\\
			1.9	0.447439639059244\\
			1.95	0.269527012370833\\
			2	-0.0961884190864275\\
			2.05	-0.3172016640734\\
			2.1	-0.406323498758089\\
			2.15	-0.00580481359767404\\
			2.2	0.0993960061825163\\
			2.3	0.166172365072665\\
			2.35	0.205138457683637\\
			2.4	0.288793818184997\\
			2.45	-0.0112538103047015\\
			2.5	-0.0913377164763052\\
			2.55	-0.133770664345363\\
			2.6	-0.172426751992086\\
			2.65	-0.232008199426181\\
			2.7	-0.312521889175855\\
			2.75	-0.383167890625936\\
			2.8	-0.432342420510793\\
			2.85	-0.437903571993315\\
			2.9	-0.394584430782972\\
			2.95	-0.0808839026003625\\
			3	0.490986879368496\\
			3.05	0.622073597030406\\
			3.1	0.680030668708291\\
			3.15	0.69745419289247\\
			3.2	0.687809834261907\\
			3.25	0.659019050518443\\
			3.3	0.60253017670771\\
			3.35	0.510908689999008\\
			3.4	0.33079401909538\\
			3.45	-0.167703086583847\\
			3.5	-0.482811221732542\\
			3.55	-0.57777781066048\\
			3.6	-0.631373355242577\\
			3.65	-0.656208497547693\\
			3.7	-0.652490719020765\\
			3.75	-0.627237344069252\\
			3.8	-0.5657294323317\\
			3.85	-0.46938415908742\\
			3.9	-0.305503908870058\\
			3.95	0.0340912162834233\\
			4	0.335447140875546\\
			4.05	0.428377037890671\\
			4.1	0.47208299140776\\
			4.15	0.492019541551475\\
			4.2	0.49470033422131\\
			4.25	0.483012897992827\\
			4.3	0.45099396721964\\
			4.35	0.403120161687346\\
			4.4	0.335126932253283\\
			4.45	0.194469487175775\\
			4.5	-0.0979897883986656\\
			4.55	-0.330867902851805\\
			4.6	-0.389384538466762\\
			4.65	-0.40694107676994\\
			4.7	-0.406305918243987\\
			4.75	-0.391197899322702\\
			4.8	-0.358203966499094\\
			4.85	-0.299300762868299\\
			4.9	-0.190020512855403\\
			4.95	0.0155102029901943\\
			5	0.265149233751602\\
			5.05	0.402940477678392\\
			5.1	0.455393900108499\\
			5.15	0.479266706019922\\
			5.2	0.483676438787887\\
			5.25	0.470710513325415\\
			5.3	0.437432134926901\\
			5.35	0.370746590877458\\
			5.4	0.22989486277462\\
			5.45	-0.0272449672628206\\
			5.5	-0.250930789684366\\
			5.55	-0.314505832091267\\
			5.6	-0.339391038316458\\
			5.65	-0.349789824811012\\
			5.7	-0.344145168124099\\
			5.75	-0.321235060024003\\
			5.8	-0.271069810966939\\
			5.85	-0.179520713206349\\
			5.9	-0.029259167098882\\
			5.95	0.168571766523574\\
			6	0.335597812599671\\
			6.05	0.423008137127265\\
			6.1	0.460533287146349\\
			6.15	0.473549881029612\\
			6.2	0.469372619417779\\
			6.25	0.450726994624221\\
			6.3	0.415457886253302\\
			6.35	0.358775202555043\\
			6.4	0.25248625184128\\
			6.45	0.0428024235373616\\
			6.5	-0.201487408923537\\
			6.55	-0.282401828543206\\
			6.6	-0.313096056226055\\
			6.65	-0.296276710796381\\
			6.7	-0.308585136141472\\
			6.75	-0.322635994230339\\
			6.8	-0.408476842018912\\
			6.85	-0.396127758125953\\
			6.9	-0.358007009711317\\
			6.95	-0.236883807193172\\
			7	-0.0378851511199354\\
			7.05	0.203539250876268\\
			7.1	0.357920687840249\\
			7.15	0.392050969303277\\
			7.2	0.405308660817497\\
			7.25	0.386743552490771\\
			7.3	0.360470475936088\\
			7.35	0.267619188280035\\
			7.4	0.162283203273121\\
			7.5	-0.202984756545494\\
			7.55	-0.274253928545239\\
			7.6	-0.254161824748881\\
			7.65	-0.19811767976311\\
			7.7	-0.196597607983703\\
			7.75	-0.237357112806116\\
			7.8	-0.339755376904217\\
			7.85	-0.322894298044739\\
			7.9	-0.228702421653136\\
			7.95	-0.0262802907967732\\
			8	0.190567612768438\\
			8.05	0.311771432211804\\
			8.1	0.363995780245837\\
			8.15	0.367419934269352\\
			8.2	0.368643332461607\\
			8.25	0.347673912788398\\
			8.3	0.31910279445556\\
			8.35	0.235954594354613\\
			8.4	0.124409959045101\\
			8.45	-0.0713505347662888\\
			8.5	-0.253409127875345\\
			8.55	-0.317494335438182\\
			8.6	-0.324013507774097\\
			8.75	-0.293525637383507\\
			8.8	-0.269369225457837\\
			8.85	-0.197811417576968\\
			8.9	-0.0446988276910716\\
			8.95	0.150383805500866\\
			9	0.280100092528931\\
			9.05	0.323406065926354\\
			9.1	0.343711170802482\\
			9.15	0.350462236660467\\
			9.2	0.342995853776793\\
			9.25	0.321706348523563\\
			9.3	0.280313799656341\\
			9.35	0.208041487714441\\
			9.4	0.0778834721936352\\
			9.45	-0.112283826671781\\
			9.5	-0.259312985343286\\
			9.55	-0.310742962433\\
			9.6	-0.332720360701392\\
			9.65	-0.339905007010136\\
			9.7	-0.333262265258695\\
			9.75	-0.31281846113928\\
			9.8	-0.27381054399817\\
			9.85	-0.206604282979464\\
			9.9	-0.085169955612928\\
			9.95	0.097936378460231\\
			10	0.250731112793458\\
			10.05	0.30393460589865\\
			10.1	0.325029870497071\\
			10.15	0.331943391453189\\
			10.2	0.325732929712604\\
			10.25	0.306610246668297\\
			10.3	0.268236071094529\\
			10.35	0.200288471434407\\
			10.4	0.0787675913726886\\
			10.45	-0.100036949017129\\
			10.5	-0.245596017725244\\
			10.55	-0.297589997072544\\
			10.6	-0.319110952142294\\
			10.65	-0.326505828909452\\
			10.7	-0.32064096335394\\
			10.75	-0.30188590083482\\
			10.8	-0.26374470528221\\
			10.85	-0.196057489281714\\
			10.9	-0.076078156774587\\
			10.95	0.0988465128053519\\
			11	0.24070750313887\\
			11.05	0.291884875439225\\
			11.1	0.313434582105836\\
			11.15	0.320806186616768\\
			11.2	0.31504770891029\\
			11.25	0.29649540079453\\
			11.3	0.258944071506704\\
			11.35	0.192598710604992\\
			11.4	0.0755800055435145\\
			11.45	-0.0952377201000338\\
			11.5	-0.235562242616592\\
			11.55	-0.286547549796646\\
			11.6	-0.30805121768681\\
			11.65	-0.315452916409569\\
			11.7	-0.309805258623165\\
			11.75	-0.291513120852946\\
			11.8	0.230700536331057\\
			11.85	-0.183995874956581\\
			11.9	-0.725920422973225\\
			11.95	-1.10246291395059\\
			12	-1.1070540530089\\
			12.05	-0.792773957060909\\
			12.1	-0.331510403195168\\
			12.15	0.160743722434578\\
			12.2	0.98245545167479\\
			12.25	1.44759185822748\\
			12.3	1.44123346722991\\
			12.35	1.5737458312354\\
			12.4	1.08197029789563\\
			12.45	0.0824110257531654\\
			12.5	-0.707598338683546\\
			12.55	-1.30082835712429\\
			12.6	-1.3891060086704\\
			12.65	-1.43111866100256\\
			12.7	-1.32327385702652\\
			12.75	-0.649747464267021\\
			12.8	-0.102294081045681\\
			12.85	-0.170699371948233\\
			12.9	-0.122536457475544\\
			12.95	-0.0653755936584073\\
			13	0.0613799035690263\\
			13.05	0.101679110666373\\
			13.1	0.120595758046264\\
			13.2	0.116812972270726\\
			13.25	0.065047037431647\\
			13.3	-0.0707076087121372\\
			13.35	-0.110420234517418\\
			13.4	-0.107577221432576\\
			13.45	-0.0906049611887045\\
			13.5	-0.0699245934848474\\
			13.55	-0.0572800477458131\\
			13.6	0.0541390870384255\\
			13.65	0.0877936594385442\\
			13.7	0.0522417399267265\\
			13.75	0.00200255863832588\\
			13.8	0.0193269817011359\\
			13.85	0.0621120299943279\\
			13.9	0.0782281265274172\\
			13.95	0.0802690800823012\\
			14	0.0634204077811003\\
			14.05	0.0495739026032371\\
			14.15	-0.0374212109027479\\
			14.2	-0.0488494370545904\\
		};
		\addlegendentry{MDeePC}
		
		\addplot [color=magenta, solid, very thick]
		table[row sep=crcr]{%
			0	0\\
			1.35	-4.02877759597686e-06\\
			1.4	-0.19009992864348\\
			1.45	-0.366304178838522\\
			1.5	-0.508997769715112\\
			1.55	-0.603560183894462\\
			1.6	-0.639936984843468\\
			1.65	-0.613682561509618\\
			1.7	-0.526480480061938\\
			1.75	-0.386066376705015\\
			1.8	-0.201605926984847\\
			1.85	-0.0021828121037526\\
			1.9	-0.000520093380568909\\
			3.8	-0.000819841515591335\\
			4.05	-0.000870328701703116\\
			4.55	0.000945946957591204\\
			5.05	-0.000992834125229436\\
			5.55	0.00104526254626691\\
			6	-0.00142317808670356\\
			6.55	0.00108778693274481\\
			7	-0.00150260375597\\
			7.5	0.00152123229319656\\
			8	-0.00152969628864241\\
			8.45	0.00200658861234615\\
			8.95	-0.00204292597492461\\
			9.45	0.00206574265728854\\
			9.95	-0.00208514383947289\\
			10.4	0.00251264762441394\\
			11	-0.00141719417540465\\
			11.4	0.00259812345193211\\
			11.75	-0.00120938385964564\\
			11.8	0.483108591488058\\
			11.85	0.00252331418020368\\
			11.9	-0.00751742454601612\\
			11.95	-2.12808197925796e-05\\
			12.4	6.42968352515538e-06\\
			14.2	-1.07790178560663e-07\\
		};
		\addlegendentry{MFDOOM}
		
		\addplot [color=black, forget plot]
		table[row sep=crcr]{%
			1.25	-2\\
			1.25	2\\
		};
		\addplot [color=white!15!black, forget plot]
		table[row sep=crcr]{%
			11.75	-2\\
			11.75	2\\
		};
		
		\addplot[area legend, draw=none, fill=white!90!black, fill opacity=0.3, forget plot]
		table[row sep=crcr] {%
			x	y\\
			1.25	-2.5\\
			11.75	-2.5\\
			11.75	2\\
			1.25	2\\
		}--cycle;
	\end{axis}
\end{tikzpicture}
\caption{Velocity tracking error comparison of ODeePC (RMSE = 2.18), MDeePC (RMSE = 0.40) and MFDOOM (RMSE = 0.09) when controlling \cref{eqn:sim} under time-varying harmonic disturbance (gray region).}
\label{fig:LTV_error_response}
\end{figure}
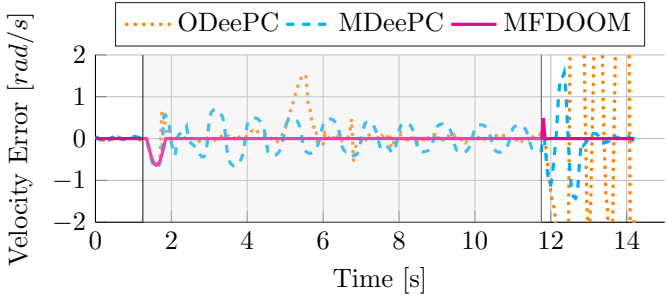

\begin{figure}[h!]\centering
\definecolor{mycolor1}{rgb}{0.85000,0.32500,0.09800}%
\definecolor{mycolor2}{rgb}{0,0.5,0}%
\begin{tikzpicture}
	
	\begin{axis}[%
		width=0.85\columnwidth,   %
		height=0.25\columnwidth,  %
		scale only axis,
		unbounded coords=jump,
		xmin=0,
		xmax=15,
		ymin=-2,
		ymax=2,
		xlabel={Time [s]},
		ylabel={Prediction Error $[rad/s]$},
		axis background/.style={fill=white},
		axis x line*=bottom,
		axis y line*=left,
		xmajorgrids,
		ymajorgrids,
		legend style={
			at={(0.5,1.05)},
			anchor=south,
			legend columns=-1,
			legend cell align=left,
			align=left,
			draw=white!15!black
		}
		]
		\addplot [color=black, thick]
		table[row sep=crcr]{%
			0	-5.32907051820075e-15\\
			1.35	-2.63646882103785e-11\\
			1.4	-0.190094833134188\\
			1.45	-0.366298154125637\\
			1.5	-0.508991097113928\\
			1.55	-0.603553028534909\\
			1.6	-0.639929706102521\\
			1.65	-0.613675102711495\\
			1.7	-0.526473051382611\\
			1.75	-0.386059157851395\\
			1.8	-0.205541929903834\\
			1.9	0.204257853354767\\
			1.95	0.393432220358438\\
			2	0.546419040200643\\
			2.05	0.64761091659768\\
			2.1	0.686304843529946\\
			2.15	0.657827968364181\\
			2.2	0.564081654881857\\
			2.25	0.413442107178039\\
			2.3	0.220018791172821\\
			2.4	-0.218442403339781\\
			2.45	-0.4205665151269\\
			2.5	-0.583846985677203\\
			2.55	-0.691668804701807\\
			2.6	-0.732679981201871\\
			2.65	-0.701980834010907\\
			2.7	-0.601690258383837\\
			2.75	-0.440825056408974\\
			2.8	-0.23449565225496\\
			2.9	0.232626953580008\\
			2.95	0.447700809886106\\
			3	0.621274931159189\\
			3.05	0.73572669278928\\
			3.1	0.779055118847058\\
			3.15	0.746133699615054\\
			3.2	0.639298861912581\\
			3.25	0.468208005618708\\
			3.3	0.248972513319242\\
			3.4	-0.246811503822459\\
			3.45	-0.474835104676556\\
			3.5	-0.65870287661371\\
			3.55	-0.779784580867899\\
			3.6	-0.825430256547524\\
			3.65	-0.790286565287191\\
			3.7	-0.676907465396157\\
			3.75	-0.495590954845394\\
			3.8	-0.263449374364994\\
			3.9	0.260996054079515\\
			3.95	0.501969399454026\\
			4	0.696130822116407\\
			4.05	0.823842469009247\\
			4.1	0.871805394206026\\
			4.15	0.834439430953614\\
			4.2	0.71451606890229\\
			4.25	0.522973904038858\\
			4.3	0.277926235463505\\
			4.4	-0.275180604313954\\
			4.45	-0.5291036942275\\
			4.5	-0.733558767593836\\
			4.55	-0.867900357114467\\
			4.6	-0.918180531883449\\
			4.65	-0.878592296606314\\
			4.7	-0.752124672410879\\
			4.75	-0.550356853275099\\
			4.8	-0.292403096530045\\
			4.9	0.289365154552907\\
			4.95	0.556237989002755\\
			5	0.77098671306582\\
			5.05	0.911958245094095\\
			5.1	0.9645556694988\\
			5.15	0.922745162213269\\
			5.2	0.789733275873825\\
			5.25	0.577739802409877\\
			5.3	0.306879957475386\\
			5.4	-0.303549704847617\\
			5.45	-0.58337228378587\\
			5.5	-0.808414658545164\\
			5.55	-0.956016133319055\\
			5.6	-1.01093080722564\\
			5.65	-0.966898027942925\\
			5.7	-0.82734187942164\\
			5.75	-0.605122751716202\\
			5.8	-0.32135681867336\\
			5.9	0.317734255030631\\
			5.95	0.610506578545385\\
			6	0.845842603995182\\
			6.05	1.0000740214159\\
			6.1	1.05730594487202\\
			6.15	1.01105089353771\\
			6.2	0.864950482920083\\
			6.25	0.632505700907394\\
			6.3	0.335833679729028\\
			6.4	-0.331918805300226\\
			6.45	-0.637640873334988\\
			6.5	-0.883270549519331\\
			6.55	-1.04413190954023\\
			6.6	-1.10368108256926\\
			6.65	-1.05520375920685\\
			6.7	-0.902559086438105\\
			6.75	-0.659888650147055\\
			6.8	-0.350310540812565\\
			6.9	0.346103355527035\\
			6.95	0.664775168100384\\
			7	0.920698494973019\\
			7.05	1.08818979761779\\
			7.1	1.15005622022434\\
			7.15	1.09935662484484\\
			7.2	0.940167689930826\\
			7.25	0.687271599352574\\
			7.3	0.364787401836368\\
			7.4	-0.360287905776993\\
			7.45	-0.691909462885098\\
			7.5	-0.958126440468591\\
			7.55	-1.13224768573631\\
			7.6	-1.1964313579054\\
			7.65	-1.1435094904963\\
			7.7	-0.977776293439781\\
			7.75	-0.714654548578205\\
			7.8	-0.379264262945977\\
			7.9	0.374472455997733\\
			7.95	0.719043757678522\\
			8	0.995554385930479\\
			8.05	1.17630557382892\\
			8.1	1.24280649557232\\
			8.15	1.18766235613268\\
			8.2	1.01538489693976\\
			8.25	0.742037497780865\\
			8.3	0.393741124015985\\
			8.4	-0.388657006244072\\
			8.45	-0.746178052399655\\
			8.5	-1.03298233141451\\
			8.55	-1.22036346194185\\
			8.6	-1.28918163320376\\
			8.65	-1.23181522177134\\
			8.7	-1.05299350043647\\
			8.75	-0.76942044699423\\
			8.8	-0.408217985082693\\
			8.9	0.402841556492323\\
			8.95	0.773312347191736\\
			9	1.0704102769154\\
			9.05	1.2644213500686\\
			9.1	1.33555677090062\\
			9.15	1.27596808743027\\
			9.2	1.09060210394584\\
			9.25	0.796803396204917\\
			9.3	0.422694846152284\\
			9.4	-0.417026106748965\\
			9.45	-0.80044664196385\\
			9.5	-1.10783822235265\\
			9.55	-1.30847923812632\\
			9.6	-1.38193190859022\\
			9.65	-1.32012095307442\\
			9.7	-1.12821070745077\\
			9.75	-0.824186345427043\\
			9.8	-0.4371717072196\\
			9.9	0.43121065697356\\
			9.95	0.827580936745489\\
			10	1.14526616786322\\
			10.05	1.35253712620936\\
			10.1	1.42830704621427\\
			10.15	1.36427381871511\\
			10.2	1.16581931096962\\
			10.25	0.85156929463559\\
			10.3	0.451648568288872\\
			10.4	-0.445395207222058\\
			10.45	-0.854715231511992\\
			10.5	-1.18269411330297\\
			10.55	-1.39659501432152\\
			10.6	-1.47468218391191\\
			10.65	-1.40842668435079\\
			10.7	-1.20342791449393\\
			10.75	-0.878952243858741\\
			10.8	-0.466125429370345\\
			10.9	0.459579757457707\\
			10.95	0.881849526253331\\
			11	1.22012205876129\\
			11.05	1.44065290242681\\
			11.1	1.52105732157779\\
			11.15	1.45257954999603\\
			11.2	1.24103651795946\\
			11.25	0.906335193069021\\
			11.3	0.480602290432621\\
			11.4	-0.473764307711189\\
			11.45	-0.908983821077847\\
			11.5	-1.25755000429848\\
			11.55	-1.48471079051117\\
			11.6	-1.56743245922843\\
			11.65	-1.49673241564722\\
			11.7	-1.27864512145742\\
			11.75	-0.933718142287921\\
			11.8	-0.00991174156306229\\
			11.85	-0.000105216570798206\\
			12.6	1.77635683940025e-15\\
			14.2	0\\
		};
		\addlegendentry{DeePC}
		
		\addplot [color=mycolor1,thick,dashed]
		table[row sep=crcr]{%
			0	0\\
			1.7	0\\
			1.75	-0.924972822192245\\
			1.8	-0.205325843394785\\
			1.9	0.20444690471065\\
			1.95	0.393608807758321\\
			2	0.546573899940775\\
			2.05	0.647733125332213\\
			2.1	0.686383500947375\\
			2.15	0.65785349084989\\
			2.2	0.564047129530518\\
			2.25	0.413344945473439\\
			2.3	0.219862598713922\\
			2.4	-0.218675289724453\\
			2.45	-0.420803943981898\\
			2.5	-0.584062695830674\\
			2.55	-0.691838306482595\\
			2.6	-0.732783202801039\\
			2.65	-0.702003722699228\\
			2.7	-0.601625845511\\
			2.75	-0.440674398737315\\
			2.8	-0.234268609003255\\
			2.9	0.23294257006927\\
			2.95	0.448015629318711\\
			3	0.621555764578705\\
			3.05	0.735942916941424\\
			3.1	0.77918202871253\\
			3.15	0.746154898549626\\
			3.2	0.639207802720213\\
			3.25	0.468008892413609\\
			3.3	0.24868056340326\\
			3.4	-0.247204029669122\\
			3.45	-0.475221885058142\\
			3.5	-0.659044040830501\\
			3.55	-0.780043888574731\\
			3.6	-0.825578977076781\\
			3.65	-0.790306366891791\\
			3.7	-0.676792253723832\\
			3.75	-0.495347741274934\\
			3.8	-0.26309815384926\\
			3.9	0.261459176421655\\
			3.95	0.502422287192971\\
			4	0.696527396690016\\
			4.05	0.824141338017171\\
			4.1	0.871974200648959\\
			4.15	0.83445813748062\\
			4.2	0.71437901158107\\
			4.25	0.522690611179677\\
			4.3	0.277520990759475\\
			4.4	-0.275708483864411\\
			4.45	-0.529617387497149\\
			4.5	-0.734006209272806\\
			4.55	-0.868235650527826\\
			4.6	-0.918368016062391\\
			4.65	-0.87861037492295\\
			4.7	-0.751968049882869\\
			4.75	-0.550037229589572\\
			4.8	-0.291948945228915\\
			4.9	0.289951631859285\\
			4.95	0.556806633839157\\
			5	0.771480377863432\\
			5.05	0.912326432963166\\
			5.1	0.964759639821818\\
			5.15	0.922761964370613\\
			5.2	0.789558035717558\\
			5.25	0.577386451410639\\
			5.3	0.306380231996839\\
			5.4	-0.304190839501038\\
			5.45	-0.583992365986248\\
			5.5	-0.808951571189221\\
			5.55	-0.956415322668901\\
			5.6	-1.01115091030714\\
			5.65	-0.966914820479529\\
			5.7	-0.827150782909184\\
			5.75	-0.604739448437421\\
			5.8	-0.320816158188972\\
			5.9	0.318424947847069\\
			5.95	0.611173003295722\\
			6	0.846418247853016\\
			6.05	1.00050068173029\\
			6.1	1.05753963241787\\
			6.15	1.01106622400455\\
			6.2	0.86474320462097\\
			6.25	0.632093253137334\\
			6.3	0.335253844661127\\
			6.4	-0.332656622603045\\
			6.45	-0.638351675964508\\
			6.5	-0.88388365253066\\
			6.55	-1.04458569799366\\
			6.6	-1.10392901155949\\
			6.65	-1.055219399301\\
			6.7	-0.902338482424941\\
			6.75	-0.659450935313039\\
			6.8	-0.349696214518023\\
			6.9	0.346882932890381\\
			6.95	0.665525150389657\\
			7	0.921344277216932\\
			7.05	1.08866662320656\\
			7.1	1.15031558840258\\
			7.15	1.09937122238481\\
			7.2	0.939933903820592\\
			7.25	0.686809915498616\\
			7.3	0.364140527125066\\
			7.4	-0.361107306521422\\
			7.45	-0.692697258907847\\
			7.5	-0.958804383492176\\
			7.55	-1.13274786661374\\
			7.6	-1.19670299167264\\
			7.65	-1.14352429004225\\
			7.7	-0.97753082536301\\
			7.75	-0.714170811190563\\
			7.8	-0.378587483323667\\
			7.9	0.375328042333892\\
			7.95	0.719865385137938\\
			8	0.996260498685272\\
			8.05	1.17682558932316\\
			8.1	1.24308802386448\\
			8.15	1.18767682077346\\
			8.2	1.01512876125837\\
			8.25	0.741533910589686\\
			8.3	0.393037100737727\\
			8.4	-0.389547309868711\\
			8.45	-0.74703269762543\\
			8.5	-1.03371592095574\\
			8.55	-1.22090313389041\\
			8.6	-1.2894738642649\\
			8.65	-1.23182945170779\\
			8.7	-1.05272727750589\\
			8.75	-0.768897867328812\\
			8.8	-0.407487368206276\\
			8.9	0.403763476157218\\
			8.95	0.774196557015202\\
			9	1.07116863152991\\
			9.05	1.26497852934814\\
			9.1	1.33585733545495\\
			9.15	1.27598187571766\\
			9.2	1.09032650951023\\
			9.25	0.796263376891874\\
			9.3	0.42194051486182\\
			9.4	-0.417977307846007\\
			9.45	-0.801359067216474\\
			9.5	-1.10862053017921\\
			9.55	-1.30905324800616\\
			9.6	-1.38224098825373\\
			9.65	-1.32013453071218\\
			9.7	-1.12792637824541\\
			9.75	-0.823629917286235\\
			9.8	-0.436395164811227\\
			9.9	0.432190297160497\\
			9.95	0.828519684200346\\
			10	1.14607037068449\\
			10.05	1.35312705711801\\
			10.1	1.42862415586183\\
			10.15	1.36428715513288\\
			10.2	1.16552688681901\\
			10.25	0.850997487160489\\
			10.3	0.450850942480661\\
			10.4	-0.446400897633907\\
			10.45	-0.855678942264159\\
			10.5	-1.18351937096071\\
			10.55	-1.39719956102266\\
			10.6	-1.47500664198693\\
			10.65	-1.40843976724555\\
			10.7	-1.20312794321769\\
			10.75	-0.878366495023352\\
			10.8	-0.465308589568059\\
			10.9	0.460610341441521\\
			10.95	0.882836915769685\\
			11	1.22096711098437\\
			11.05	1.44127130036751\\
			11.1	1.52138863501126\\
			11.15	1.45259237206514\\
			11.2	1.24072946146416\\
			11.25	0.905736311156893\\
			11.3	0.479767008994513\\
			11.4	-0.474817852222323\\
			11.45	-0.909992906652461\\
			11.5	-1.25841326539339\\
			11.55	-1.4853421857676\\
			11.6	-1.56777030598642\\
			11.65	-1.49674495283964\\
			11.7	-1.27833141017003\\
			11.75	-0.933106937663469\\
			11.8	-0.494226228775098\\
			11.85	0.922250372889103\\
			11.9	0.0097900171364369\\
			11.95	0.000103923745932377\\
			12.7	5.17922149612104e-08\\
			14.2	3.557154570899e-11\\
		};
		\addlegendentry{$d^*(0)$}
		
		\addplot [color=mycolor2,thick,dotted]
		table[row sep=crcr]{%
			0	0\\
			1.9	0\\
			1.95	-0.956742204150061\\
			2	0.548164157042503\\
			2.05	0.64923305949447\\
			2.1	0.687751385841192\\
			2.15	0.658998932230119\\
			2.2	0.564871318469393\\
			2.25	0.413756867990543\\
			2.3	0.219789279306573\\
			2.4	-0.219816668048876\\
			2.45	-0.422435765829746\\
			2.5	-0.58607164172043\\
			2.55	-0.694042022130333\\
			2.6	-0.734947097141843\\
			2.65	-0.703876160554804\\
			2.7	-0.60297845529405\\
			2.75	-0.44133243223682\\
			2.8	-0.234126090786049\\
			2.9	0.234705753742841\\
			2.95	0.450443589500354\\
			3	0.624446829686228\\
			3.05	0.739027140648208\\
			3.1	0.782145496816755\\
			3.15	0.748676774593662\\
			3.2	0.64100163111249\\
			3.25	0.468857638620207\\
			3.3	0.248459256067589\\
			3.4	-0.24952692508192\\
			3.45	-0.478368928391452\\
			3.5	-0.662737698853304\\
			3.55	-0.783934151867062\\
			3.6	-0.829275736657062\\
			3.65	-0.793421899259494\\
			3.7	-0.678986477804235\\
			3.75	-0.496367341923683\\
			3.8	-0.262803526093608\\
			3.9	0.264289541562837\\
			3.95	0.506222244062322\\
			4	0.700951728703266\\
			4.05	0.828767310377248\\
			4.1	0.876341314411093\\
			4.15	0.838116515360472\\
			4.2	0.716939221665125\\
			4.25	0.52386669977329\\
			4.3	0.277160806681426\\
			4.4	-0.278998552422751\\
			4.45	-0.534009871315574\\
			4.5	-0.739095294583953\\
			4.55	-0.873532330729788\\
			4.6	-0.923348474360299\\
			4.65	-0.88276707047128\\
			4.7	-0.754863803190526\\
			4.75	-0.551358554247154\\
			4.8	-0.291532143803538\\
			4.9	0.29365717904494\\
			4.95	0.561737003228062\\
			5	0.777171436738495\\
			5.05	0.918230122648023\\
			5.1	0.970294477643048\\
			5.15	0.927367599821949\\
			5.2	0.792757261446789\\
			5.25	0.578835987618953\\
			5.3	0.305905837156084\\
			5.4	-0.308282581555995\\
			5.45	-0.589418603308186\\
			5.5	-0.815200843953724\\
			5.55	-0.962886685918011\\
			5.6	-1.01720267433427\\
			5.65	-0.971942205322568\\
			5.7	-0.830637065301165\\
			5.75	-0.606315774809415\\
			5.8	-0.320300444491359\\
			5.9	0.322857045506005\\
			5.95	0.617042861282089\\
			6	0.853169644144396\\
			6.05	1.00748010703489\\
			6.1	1.06405879664549\\
			6.15	1.01647140344974\\
			6.2	0.868482845889554\\
			6.25	0.633775621649242\\
			6.3	0.33468748250891\\
			6.4	-0.337417521743989\\
			6.45	-0.644645048279788\\
			6.5	-0.89111167477374\\
			6.55	-1.0520489938949\\
			6.6	-1.11089374348895\\
			6.65	-1.06099030555872\\
			6.7	-0.906330843599207\\
			6.75	-0.661251321208901\\
			6.8	-0.349102280190941\\
			6.9	0.351929746167649\\
			6.95	0.672191832758115\\
			7	0.928994620057985\\
			7.05	1.09655877368455\\
			7.1	1.15767316523704\\
			7.15	1.1054603922088\\
			7.2	0.944138510719322\\
			7.25	0.688696984374706\\
			7.3	0.363503203448619\\
			7.4	-0.366425488897692\\
			7.45	-0.699713238852151\\
			7.5	-0.966849997538317\\
			7.55	-1.1410438849411\\
			7.6	-1.20443588068468\\
			7.65	-1.14992506677681\\
			7.7	-0.981954218366191\\
			7.75	-0.716162344243385\\
			7.8	-0.377929546646277\\
			7.9	0.380892060840644\\
			7.95	0.72720396481934\\
			8	1.00467035011165\\
			8.05	1.18549292250942\\
			8.1	1.25116116471468\\
			8.15	1.19435224409981\\
			8.2	1.01973520178935\\
			8.25	0.743600697404949\\
			8.3	0.392343729930401\\
			8.4	-0.395334143933256\\
			8.45	-0.754660180898972\\
			8.5	-1.04245567081393\\
			8.55	-1.22991974911678\\
			8.6	-1.29787241791583\\
			8.65	-1.23877704525408\\
			8.7	-1.05751978553134\\
			8.75	-0.77105036535054\\
			8.8	-0.40677950845533\\
			8.9	0.409760441327991\\
			8.95	0.782100582432843\\
			9	1.08022925151637\\
			9.05	1.27431112520348\\
			9.1	1.344554332082\\
			9.15	1.28317161213867\\
			9.2	1.0952834418824\\
			9.25	0.798486978034971\\
			9.3	0.421201975578649\\
			9.4	-0.424168392949083\\
			9.45	-0.809515818999763\\
			9.5	-1.11796778458012\\
			9.55	-1.31868634167019\\
			9.6	-1.39121322730031\\
			9.65	-1.32755641052786\\
			9.7	-1.1330459137557\\
			9.75	-0.825926133472496\\
			9.8	-0.435636555559274\\
			9.9	0.438560566428313\\
			9.95	0.83691289433319\\
			10	1.15568573109207\\
			10.05	1.3630390110568\\
			10.1	1.43786369579\\
			10.15	1.37192577119646\\
			10.2	1.17079403067798\\
			10.25	0.853364129010334\\
			10.3	0.450075398477203\\
			10.4	-0.452937673387931\\
			10.45	-0.864292512183008\\
			10.5	-1.19338729031448\\
			10.55	-1.40737359010122\\
			10.6	-1.48449100579992\\
			10.65	-1.41628458748455\\
			10.7	-1.20853906959807\\
			10.75	-0.880796430389147\\
			10.8	-0.464515777547261\\
			10.9	0.46730139562813\\
			10.95	0.891651340224284\\
			11	1.23106576985987\\
			11.05	1.45168857367402\\
			11.1	1.53110700267479\\
			11.15	1.46063337745455\\
			11.2	1.24627653411795\\
			11.25	0.908227680436076\\
			11.3	0.478958805785153\\
			11.4	-0.48165328220812\\
			11.45	-0.918996094400841\\
			11.5	-1.26873226178104\\
			11.55	-1.49599159157033\\
			11.6	-1.57770418090661\\
			11.65	-1.50496511750412\\
			11.7	-1.28400279851707\\
			11.75	-0.935656451418801\\
			11.8	-0.49340553101711\\
			11.9	0.495994276019502\\
			11.95	0.946327300355406\\
			12	1.3063868434606\\
			12.031220266022	2.4\\
			nan	nan\\
			12.0608719500142	2.4\\
			12.1	0.032460024683985\\
			12.15	0.000343897334978749\\
			12.4	-1.71823579364627e-07\\
			14.2	-9.58948476181831e-11\\
		};
		\addlegendentry{$d^*(N-1)$}
		
		\addplot [color=black, forget plot]
		table[row sep=crcr]{%
			1.25	-2\\
			1.25	2\\
		};
		\addplot [color=white!15!black, forget plot]
		table[row sep=crcr]{%
			11.75	-2\\
			11.75	2\\
		};
		
		\addplot[area legend, draw=none, fill=white!90!black, fill opacity=0.3, forget plot]
		table[row sep=crcr] {%
			x	y\\
			1.25	-2\\
			11.75	-2\\
			11.75	2\\
			1.25	2\\
		}--cycle;
	\end{axis}
\end{tikzpicture}%
\caption{DeePC output prediction error due to a time-varying harmonic disturbance compared to the 0-step and $(N-1)$-step error estimations found by MFDOOM.}
\label{fig:LTV_disturbance_response}
\end{figure}
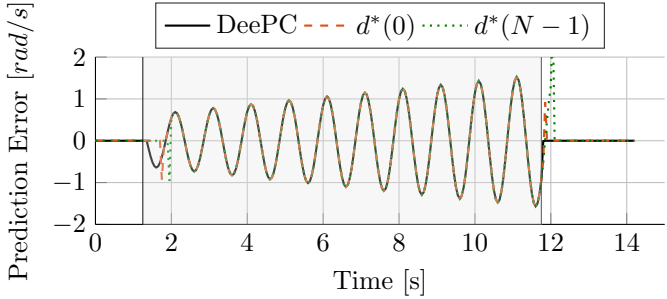
Notably, MFDOOM gains require no adjustments. In comparison, MDeePC must set $K_g=10,000$ to retain stability, which inevitably harms performance. ODeePC is stable with its previous $K_g$, but, as expected, suffers greatly from discontinuities when the disturbance is turned on/off.

\section{CONCLUSIONS}\label{sec:concl}
We presented our Model-Free Disturbance Observer with Online Modification (MFDOOM) method, which builds on DeePC \citep{DeePC} and ODeePC \citep{ODeePC}. MFDOOM introduces a dedicated, continuously updated Hankel matrix constructed from past prediction errors. This matrix serves as an implicit model of the prediction-error dynamics in DeePC and enables compensating for them online. Our analysis and simulations show that MFDOOM can outperform existing approaches when disturbances are generated by autonomous LTV systems. Although the current formulation is most suitable for such disturbances, we see significant potential for extending MFDOOM’s applicability. By further developing the adaptive gain mechanism or by combining mosaic and shifting matrices, it may become possible to handle a broader class of systems in the future.

\ack
This work was conducted while the second author was a Jane and Larry Sherman Fellow. It was further supported by the Israel Science Foundation (grant no.~2406/22) and the Bernard M. Gordon Center for Systems Engineering at the Technion--IIT.

\bibliography{ifacconf.bib}             %

@article{ODeePC,
title = {Online data-enabled predictive control},
journal = {Automatica},
volume = {138},
pages = {109926},
year = {2022},
author = {Stefanos Baros and Chin-Yao Chang and Gabriel E. Colón-Reyes and Andrey Bernstein},
}

@INPROCEEDINGS{DeePC,
  author={Coulson, Jeremy and Lygeros, John and Dörfler, Florian},
  booktitle={2019 18th European Control Conference (ECC)}, 
  title={Data-enabled predictive control: in the shallows of the {DeePC}}, 
  year={2019},
  volume={},
  number={},
  pages={307-312},
}

@article{Identifiability,
  author={Markovsky, Ivan and Dörfler, Florian},
  journal={IEEE Transactions on Automatic Control}, 
  title={Identifiability in the behavioral setting}, 
  year={2023},
  volume={68},
  number={3},
  pages={1667-1677}}

@book{MPC,
  author    = {Eduardo F. Camacho and Carlos Bordons},
  title     = {Model Predictive Control},
  edition   = {2nd},
  publisher = {Springer-Verlag London Limited},
  address   = {London},
  year      = {2000},
}

@article{Regularization,
title = {Quadratic regularization of data-enabled predictive control: theory and application to power converter experiments},
journal = {IFAC-PapersOnLine},
volume = {54},
number = {7},
pages = {192-197},
year = {2021},
note = {19th IFAC Symposium on System Identification SYSID 2021},
author = {Linbin Huang and Jianzhe Zhen and John Lygeros and Florian Dörfler},
}

@article{Persistency,
title = {A note on persistency of excitation},
journal = {Systems \& Control Letters},
volume = {54},
number = {4},
pages = {325-329},
year = {2005},
author = {Jan C. Willems and Paolo Rapisarda and Ivan Markovsky and Bart L.M. {De Moor}},
}

@article{MDeePC,
  author={Vahidi-Moghaddam, Amin and Zhang, Kaixiang and Yin, Xunyuan and Srivastava, Vaibhav and Li, Zhaojian},
  journal={IEEE Transactions on Automation Science and Engineering}, 
  title={Online reduced-order data-enabled predictive control}, 
  year={2025},
  volume={22},
  number={},
  pages={22455-22467},
}

@article{ChenDeePC2025,
	author = {Chen, Jingshan and Ebel, Henrik and Eberhard, Peter},
	journal = {Journal of Intelligent \& Robotic Systems},
	number = {2},
	pages = {32},
	title = {Practical Insights on Data-Based Robot Control: A Comparative Analysis of Data-Enabled Predictive Control and Model-Based Predictive Control},
	volume = {112},
	year = {2026}}

@article{elokda2021data,
  title={Data-enabled predictive control for quadcopters},
  author={Elokda, Ezzat and Coulson, Jeremy and Beuchat, Paul N and Lygeros, John and D{\"o}rfler, Florian},
  journal={International Journal of Robust and Nonlinear Control},
  volume={31},
  number={18},
  pages={8916--8936},
  year={2021},
  publisher={Wiley Online Library},
}

@article{huang2023robust,
  title={Robust data-enabled predictive control: Tractable formulations and performance guarantees},
  author={Huang, Linbin and Zhen, Jianzhe and Lygeros, John and D{\"o}rfler, Florian},
  journal={IEEE Transactions on Automatic Control},
  volume={68},
  number={5},
  pages={3163--3170},
  year={2023},
  publisher={IEEE},
}

@article{huang2021decentralized,
  title={Decentralized data-enabled predictive control for power system oscillation damping},
  author={Huang, Linbin and Coulson, Jeremy and Lygeros, John and D{\"o}rfler, Florian},
  journal={IEEE Transactions on Control Systems Technology},
  volume={30},
  number={3},
  pages={1065--1077},
  year={2021},
  publisher={IEEE},
}

@misc{zieglmeier2025data,
  title={Data-enabled predictive control and guidance for autonomous underwater vehicles},
  author={Zieglmeier, Sebastian and de Badyn, Mathias Hudoba and Warakagoda, Narada D and Krogstad, Thomas R and Engelstad, Paal},
  note={arXiv:2510.25309},
  year={2025},
}

\end{document}